\documentclass[11pt,oneside,a4paper,mathscr]{amsart}
\usepackage[utf8]{inputenc}

\usepackage{amsthm,amssymb,latexsym,amsmath}
\usepackage[all]{xy}
\usepackage{tikz}
\usepackage{color}
\usepackage[english]{babel}
\usepackage{spectralsequences}
\usepackage{tikz-cd}
\usepackage{tikz}
\usepackage{diagbox}
\usepackage{enumitem}
\usepackage{booktabs}
\usepackage{multirow}
\usepackage{float}

\newcommand{\C}{\mathbb{C}}
\newcommand{\R}{\mathbb{R}}
\newcommand{\HH}{\mathbb{H}}
\newcommand{\p}[1]{{\mathbb{P}^{#1}}}
\newcommand{\hp}[1]{{\mathbb{H}\mathbb{P}^{#1}}}
\newcommand{\cp}[1]{{\mathbb{C}\mathbb{P}^{#1}}}
\newcommand{\pn}{{\mathbb{P}^n}}

\newcommand{\op}[1]{{\mathcal O}_{\mathbb{P}^{#1}}}
\newcommand{\opn}{{\mathcal O}_{\mathbb{P}^n}}

\newcommand{\ch}{\operatorname{ch}}

\newcommand{\calb}{{\mathcal B}}
\newcommand{\calc}{{\mathcal C}}
\newcommand{\cald}{{\mathcal D}}
\newcommand{\cale}{{\mathcal E}}
\newcommand{\calf}{{\mathcal F}}

\newcommand{\calh}{{\mathcal H}}
\newcommand{\cali}{{\mathcal I}}

\newcommand{\calo}{{\mathcal O}}

\newcommand{\calr}{{\mathcal R}}

\newcommand{\inhom}{{\mathcal H}{\it om}}

\newcommand{\Ext}{\operatorname{Ext}}

\newcommand{\Sing}{\operatorname{Sing}}
\newcommand{\Supp}{\operatorname{Supp}}

\newcommand{\Hom}{\operatorname{Hom}}

\DeclareMathOperator{\Coh}{Coh}

\DeclareMathOperator{\End}{{End}}
\DeclareMathOperator{\coker}{coker}
\DeclareMathOperator{\im}{im}

\DeclareMathOperator{\rk}{{rk}}

\DeclareMathOperator{\Pic}{{Pic}}

\newcommand{\lra}{\longrightarrow}
\newcommand{\into}{\hookrightarrow}
\newcommand{\onto}{\twoheadrightarrow}

\newtheorem{theorem}{Theorem}[section]

\newtheorem{proposition}[theorem]{Proposition}

\newtheorem{lemma}[theorem]{Lemma}
\newtheorem{corollary}[theorem]{Corollary}

\theoremstyle{definition}
\newtheorem{remark}[theorem]{Remark}
\newtheorem{example}[theorem]{Example}
\newtheorem{definition}[theorem]{{\bf Definition}}
\usepackage{yfonts}

\title{Instantons: the next frontier}

\author{Gaia Comaschi}
\address{IMECC - UNICAMP \\ Departamento de Matem\'atica \\
Rua S\'ergio  Buarque de Holanda, 651\\ 13083-970 Campinas-SP, Brazil}
\email{gaia.comaschi@gmail.com}

\author{Marcos Jardim}
\address{IMECC - UNICAMP \\ Departamento de Matem\'atica \\
Rua S\'ergio  Buarque de Holanda, 651\\ 13083-970 Campinas-SP, Brazil}
\email{jardim@ime.unicamp.br}

\author{Cristian Martinez}
\address{School of Engineering, Science and Technology \\Universidad del Rosario\\
Carrera 6 No. 12C-16, 111711, Bogot\'a, Colombia}
\email{cristianm.martinez@urosario.edu.co}

\author{Dapeng Mu}
\address{IMECC - UNICAMP \\ Departamento de Matem\'atica \\
Rua S\'ergio  Buarque de Holanda, 651\\ 13083-970 Campinas-SP, Brazil}
\email{dapeng@unicamp.br}

\date{}

\begin{document}

\maketitle

\begin{abstract}
Instantons, emerged in particle physics, have been intensely studied since the 1970's and had an enormous impact in mathematics since then. In this paper, we focus on one particular way in which mathematical physics has guided the development of algebraic geometry in the past 40+ years. To be precise, we examine how the notion of \textit{(mathematical) instanton bundles} in algebraic geometry has evolved from a class of vector bundles over $\cp3$ both to a class of torsion free sheaves on projective varieties of arbitrary dimension, and to a class of objects in the derived category of Fano threefolds. The original results contained in this survey focus precisely on the latter direction; in particular, we prove that the classical rank 2 instanton bundles over the projective space are indeed instanton objects for any suitable chamber in the space of Bridgeland stability conditions.
\end{abstract}

\tableofcontents

\section{Introduction}

Let $(M,g)$ be a 4-dimensional, oriented Riemannian manifold, and let $E\to M$ be a complex vector bundle over $M$. A \textit{connection} on $E$ is a $\C$-linear map
$$ \nabla ~\colon~ \Gamma(E) \longrightarrow \Gamma(E)\otimes\Omega^1_M $$
satisfying the \textit{Leibniz rule}: given a smooth function $f\in C^{\infty}(M)$ and a global section $\sigma\in\Gamma(E)$, we have
$$ \nabla(f\cdot\sigma) = f\cdot \nabla(\sigma) + \sigma\otimes df; $$
here, $\Gamma(E)$ denotes the $C^{\infty}(M)$-module of smooth global sections of the vector bundle $E$, and $\Omega^p_M$ denotes the space of smooth $p$-forms on $M$.

The composition $F_\nabla:=\nabla\circ\nabla$ leads to a $C^{\infty}(M)$-linear map
$$ F_\nabla ~\colon~ \Gamma(E) \longrightarrow \Gamma(E)\otimes\Omega^2_M $$
which is called the \textit{curvature} of the connection $\nabla$. In other words, $F_\nabla$ can be regarded as a 2-form on $M$ with values in the endomorphism bundle $\End(E)$: $F_\nabla\in\Gamma(\End(E))\otimes\Omega^2_M$. 

A connection $\nabla$ is called an \textit{instanton connection} on the vector bundle $E\to M$ if $F_\nabla$ is anti-self-dual with respect to the Hodge star operator $*:\Omega^2_M\to\Omega^2_M$ (recall that $*^2=1$), that is
\begin{equation} \label{eq:asdym}
*F_\nabla = - F_\nabla.    
\end{equation}
which is known as the \textit{anti-self-dual Yang--Mills equation}.  When $M$ is not compact, one usually also imposes a finiteness condition on the total $L^2$-norm of the curvature, that is,
$$ \|F_\nabla\|_{L^2}^2 := \int_M {\rm tr}(F_\nabla\wedge*F_\nabla) < \infty. $$

Instantons have been intensely studied since the 1970's, providing the initial ideas for the rise of a new area of mathematics, namely \textit{gauge theory}, that had great influence in more traditional areas like differential topology, differential geometry, mathematical physics, representation theory, and algebraic geometry. In this paper, we focus on the latter, by explaining how the notion of \textit{(mathematical) instanton bundles} in algebraic geometry has evolved in the past 40+ years.

The starting point was the \textit{Atiyah--Ward correspondence}, which is described in Section \ref{section:AWcorrespondece} below; it was the first fundamental link between mathematical physics and algebraic geometry to be discovered, (another one is the \textit{Hitchin--Kobayashi correspondence}), and leads to the highly influential work of Atiyah, Hitchin Drinfeld and Manin \cite{ADHM}. Roughly speaking, the Atiyah--Ward correspondence transformed the differential geometric problem of finding solutions of the anti-self-dual Yang--Mills equation \ref{eq:asdym} over $M=\R^4$ into an algebraic geometric problem of constructing vector bundles over $\cp3$ with certain properties. These vector bundles were subsequently studied by mathematicians (namely, Barth, Hulek, Le Potier, Okonek, Schneider, Spindler) who coined the expression \textit{(mathematical) instanton bundles} in the 1980's. What initially was a class of vector bundles over $\cp3$ became, after the work of Salamon \cite{S84} and Donaldson \cite{DFlag}, a class of vector bundles on odd dimensional projective spaces \cite{OS} and on flag manifolds \cite{MMPL}.

In Section \ref{section:perverse} we explain why it is natural to consider a further generalization of instanton bundles on projective spaces both to arbitrary projective spaces and to non locally free sheaves. The notion of (non necessarily locally free) \textit{instanton sheaves} was introduced in \cite{J-inst}, and lead also to the notion of \textit{perverse instanton sheaves}, see Definition \ref{defn:instantons} and Definition \ref{defn:perverse}, respectively. These notions arise when one tries to compactify the moduli space of instanton bundles when regarded as an open subset either of the Gieseker moduli space of stable sheaves, or the moduli space of stable representations of a certain quiver.

In Section 4 and Section 5 we illustrate how it is possible to extend the definition of instanton to algebraic varieties beyond projective spaces.
In Section 4, we focus on the specific case of Fano threefolds of Picard rank one. Also in this setting, we start treating the case of rank 2 vector bundles: we summarize their main properties and those of their families presented by Faenzi in \cite{F}.
Once again, we are lead to take into account a wider family of sheaves which includes sheaves of arbitrary rank and that are not necessarily locally free. This is what motivated
%Comaschi and Jardim's work
\cite{CJ}, where the definition of instanton sheaves on Fano threefolds is provided and their main features are collected.
%We end the section with a brief survey of Kutznetsov's investigation 
An investigation of rank 2 instanton bundles on a Fano threefold $X$ (in particular on a Fano $X$ of index 2) has also been conducted by Kuznetsov in \cite{K}, but this time the language chosen %by the author 
to present features of instantons is the one of derived categories. Kuznetsov establishes indeed 
a correspondence between instanton bundles on $X$ and objects in the triangulated category $\calb_X:=\langle \calo_X, \calo_X(1)\rangle ^{\perp}$, named \textit{acyclic extensions} of instantons, and shows how the categorical properties of these latter ``reflect'' the sheaf-theoretical properties of instantons. We end Section 4 with a brief survey of the main results of \cite{K}.

Section \ref{section:h-inst} is dedicated to review the definition and main properties of $h$-instanton sheaves, which have been recently introduced by  Antonelli and Casnati in \cite{AC}. Their notion seeks to generalize the previous definitions of instanton sheaves to arbitrary projective schemes $X$ endowed with an ample and globally generated line bundle $\calo_X(h)$. 
For an $n$ dimensional scheme $X$, the line bundle $\calo_X(h)$ induces a finite map from $X$ to $\mathbb{P}^n$. The direct image of an ordinary $h$-instanton sheaf on $X$ 
%(see Definition \ref{defn:h-instanton}) 
by this map is indeed an instanton sheaf on $\pn$ in the sense of \cite{J-inst}, and conversely the pullback of an instanton sheaf on $\pn$ is an ordinary $h$-instanton sheaf on $X$.
One of the noticeable features of an instanton sheaf is that it can be constructed as the cohomology of a monad. An $h$-instanton sheaf on $X$ has such a monadic presentation if $X$ is ACM with respect to the line bundle $\calo_X(h)$. We then focus on some particular cases of $X$, where the monads for ordinary $h$-instantons are more neatly presented. Those monads overlap some known monadic presentations on $\pn$ and some smooth quadrics, and provide some new monadic presentations on scrolls.
Next, we review the construction of rank 2 $h$-instanton bundles on some smooth varieties of low dimension, namely curves, surfaces, Fano threefolds of Picard rank $1$ and scrolls. In particular, we compare the rank $2$ $h$-instanton bundles on a Fano threefold with the instanton bundles in Section \ref{section:Fano}. We will see that these two notions coincide in most cases.
In the end, we review an example of an $h-$instanton bundle on the image of the Segre embedding $\p1\times \p1 \times \p1 \hookrightarrow \p3$ to illustrate that $h$-instanton bundles on varieties of higher Picard rank can have several pathologies comparing to the classical instanton bundles.

Finally, Section \ref{section:bridgeland} contains the original results presented in this paper. We return to the setup of Section \ref{section:Fano} and introduce the notion of $\calc$-instanton object in the derived category of a Fano threefold of Picard number one. This notion is highly inspired by the works of Faenzi \cite{F}, Kuznetsov \cite{K}, and Comaschi--Jardim \cite{CJ}, and the examples of Bridgeland semistable objects in the projective space provided in \cite{JMM}. The idea is to find a chamber $\calc$ in the stability manifold so that semistability in this chamber alone already provides the correct cohomology vanishing conditions. In the cases of the projective space and the quadric threefold, we prove that there is a chamber $\calc$ such that $\calc$-instantons have monad-type descriptions. Examples of $\calc$-instanton objects are provided, including some that do not fit any previous definitions. Additionally, we prove that the classical rank 2 instanton bundles in the projective space are indeed $\calc$-instanton objects for any suitable chamber $\calc$.  Acyclic extensions are proved to exist for stable $\calc$-instantons on any Fano threefold of index 2, and moreover such acyclic extensions are again $\calc$-instanton objects.

\subsection*{Acknowledgments}
The authors would like to express their gratitude to Daniele Faenzi for numerous conversations on the topic of this article, and in particular for suggesting the invariance under the twisted duality functor \eqref{dualityFunctor} as the right way to fix the Chern character of an instanton object, which allowed us to find the correct $(\alpha,s)$-slice for defining $\calc$-instanton objects in Section \ref{section:bridgeland}. GC is supported by the FAPESP post-doctoral grant number 2019/21140-1 and the BEPE grant number 2022/09063-4. MJ is supported by the CNPQ grant number 302889/2018-3 and the FAPESP Thematic Project \textit{Gauge theory and Algebraic Geometry} number 2018/21391-1.  CM is supported by the FAPESP post-doctoral grant number 2020/06938-4, which is part of the FAPESP Thematic Project \textit{Gauge theory and Algebraic Geometry} number 2018/21391-1. DM is supported by the FAPESP post-doctoral grant number 2020/03499-0.

%%%%%%%%%%%%%%%%%%%%%%%%%%%%%%%%%%%%%%%%%%%%%%%%%%%%%%%%%%%%%%%%%%%
%%%%%%%%%%%%%%%%%%%%%%%%%%%%%%%%%%%%%%%%%%%%%%%%%%%%%%%%%%%%%%%%%%%

\section{Atiyah--Ward correspondence and mathematical instanton bundles}\label{section:AWcorrespondece}

We will now consider our Riemannian manifold $(M,g)$ as being 
the four dimensional sphere $S^4$ equipped with the usual round metric; recall that this is conformal to the usual euclidean metric on $\R^4\simeq S^4\setminus\{\infty\}$; by virtue of the Uhlenbeck removable singularities theorem, any instanton connection on $\R^4$ extends to an instanton on $S^4$.

Furthermore, $S^4$ can be identified with the quaternionic projective line $\hp1$ as follows: any $p\in\R^4$ is associated to the point $[p:1]\in\hp1$, while $\infty$ goes to $[1:0]\in\hp1$. One can then consider the smooth map $\tau:\cp3\to\hp1$ given by
$$ \tau[x:y:z:w] = [x+jy:z+jw] . $$
Note that the fibers of $\tau$ are isomorphic to $\cp1$: the pre-image of the point $[u+jv:1]\in\hp1$ is the line $[u:v:1:0]$, while $\tau^{-1}([1:0])=[u:v:0:0]$.

This so-called \textit{twistor map} has the following fantastic property, first noticed by Atiyah and Ward in \cite{AW}. If $(E,\nabla)$ is a complex vector bundle on $S^4\simeq\hp1$ equipped with an instanton connection, then the curvature $F_{\tau^*\nabla}$ of the pull-back connection $\tau^*\nabla$ is a 2-form with values in $\tau^*\End(E)=\End(\tau^*E)$ of type $(1,1)$. This means that $F_{\tau^*\nabla}$ induces a holomorphic structure on the pulled-back vector bundle $\tau^*E$; let us denote this by $\cale$.

Observe that $\cale$ must satisfy some obvious properties. First, its restriction to the fibers $\tau^{-1}(p)$ must always be holomorphically trivial. Second, note that
$$ \tau[x:y:z:w]=\tau[-y:x:-w:z], $$
since their images only differ by multiplication by $j$ on the left; therefore, the pulled-back bundle $\tau^*E$ must be invariant under the involution $\iota:\cp3\to\cp3$ given by $\iota[x,y,z,w]=[-y:x:-w:z]$, ie. $\iota^*\cale\simeq\cale$. % Third, the dual holomorphic bundle $\cale^*$ is induced from the dual smooth complex bundle $E^*$ equipped with the dual connection $\nabla^*$.

The non-trivial key property satisfied by $\cale$ is provided by the \textit{Penrose transform}, which provides an isomorphism between the kernel of the Laplacian operator $\Delta$ coupled to the instanton connection $\nabla$, and the sheaf cohomology group $H^1(\cp3,\cale(-2))$, where $\cale(-2)=\cale\otimes\calo_{\cp3}(-2)$ as usual. It turns out that $\ker\Delta$ is empty precisely because $S^4$ has positive scalar curvature and $F_\nabla$ is anti-self-dual. Therefore, we must have that $H^1(\cp3,\cale(-2))=0$.

The \textit{Atiyah--Ward correspondence} correspondence essentially says that the original smooth instanton connection on $S^4$ can be reconstructed from the associated holomorphic vector bundle $\cale$ satisfying the properties described above. (Since we would like to get into algebraic geometry as soon as possible, we are actually omitting many details here and this claim is an oversimplification of the actual theorems; the interested reader should look at \cite{A,WW}).

The Atiyah--Ward correspondence was the first fundamental link between mathematical physics an algebraic geometry to be discovered, and it has sparked a flurry of intense activity among algebraic geometers; Hartshorne \cite{H78} and Barth--Hulek \cite{BH} are perhaps the first algebraic geometry papers on bundles on projective spaces motivated by the works of Atiyah and collaborators (especially \cite{AW,ADHM}). However, to our knowledge, the first reference that uses the expression \textit{``instanton bundle"} is \cite[p. 370]{OSS}; in this reference, the authors define a \textit{complex instanton bundle} on $\cp3$ as a stable rank 2 holomorphic bundle $\cale$ with $c_1(\cale)=0$ satisfying $H^2(\cale(-2))=0$.

At this point (circa 1977-78), Barth noticed that every such complex instanton bundle $E$ can be realized as the cohomology of a \textit{monad}, that is, a complex of sheaves
\begin{equation} \label{monad1}
\calo_{\cp3}(-1)^{\oplus c} \stackrel{\alpha}{\longrightarrow} \calo_{\cp3}^{\oplus 2+2c} \stackrel{\beta}{\longrightarrow} \calo_{\cp3}(1)^{\oplus c}    
\end{equation}
for which $\alpha$ is injective, $\beta$ is surjective, and such that $E\simeq\ker\beta/\im\alpha$, and where $c=h^1(\cale(-1))$. Since the morphisms $\alpha$ and $\beta$ can be regarded as matrices whose entries are linear polynomials, this allowed to translate the classification of complex instanton bundles into a problem in linear algebra. This is essentially the crucial point explored in the seminal paper \cite{ADHM} by Atiyah, Drinfeld, Hitchin and Manin, where the full classification of $SU(2)$ instantons on $S^4$ was given. 

In 1984, Salamon presented a higher dimensional version of the Atiyah--Ward correspondence \cite{S84}; see also \cite{MCS}. He considers a map $\tau:\cp{2k+1}\to\hp{k}$ given by
$$ \tau[x_0:y_0:\cdots:x_k:y_k]=[x_0+jy_0:\cdots:x_k+jy_k], $$
thus generalizing the twistor map defined above; the fibers of $\tau:\cp{2k+1}\to\hp{k}$ are also isomorphic to $\cp1$. Now let $E\to\hp{k}$ be a complex vector bundle equipped with a connection $\nabla$; this is said to be a \textit{quaternionic instanton} if its curvature $F_\nabla$ is of type $(1,1)$ with respect to any choice of almost complex structure in $\hp{k}$; when $k=1$, this is equivalent to the usual definition of an instanton on a 4-dimensional manifold.

It turns out that this condition is just what is needed to prove that the curvature $F_{\tau^*\nabla}$ of the pulled-back connection on $\tau^*E$ is of type $(1,1)$ on $\cp{2k+1}$, and therefore induces a holomorphic structure on $\tau^*E$. 

This motivated the definition of \textit{mathematical instanton bundles} by Okonek and Spindler, see \cite{OS} in the following year. To be precise, a mathematical instanton bundle is a rank $2k$ holomorphic bundle $\cale$ on $\cp{2k+1}$ satisfying the following conditions
\begin{enumerate}
\item $E$ is simple, ie. $\Hom(\cale,\cale)=\C$;
\item its Chern polynomial is given by $c_t(\cale)=(1-t^2)^{-c}$;
\item it has natural cohomology in the rank $-2k-1\le p\le 0$, that is, for each $p$ in the specified rank, at most one of the cohomology groups $H^p(\cale(l))$ can be non trivial;
\item $\cale$ has \textit{trivial splitting type}, ie. $\cale|_l$ is trivial for at least one line $\ell\subset\cp{2k+1}$;
\item $\cale$ admits a symplectic structure, meaning that there exists an isomorphism $\phi:\cale\to\cale^*$ such that $\phi^*=-\phi$.
\end{enumerate}
Later, Ancona and Ottaviani noticed in \cite{AO} that conditions (2) and (3) imply condition (1). Mathematical instanton bundles were the subject of several articles in 1980's and 1990's. One fact that will be relevant later on is that conditions (2) and (3) imply that any mathematical instanton bundle is isomorphic to the cohomology of a monad
\begin{equation} \label{monad2}
\calo_{\cp{2k+1}}(-1)^{\oplus c} \stackrel{\alpha}{\longrightarrow} \calo_{\cp{2k+1}}^{\oplus 2k+2c} \stackrel{\beta}{\longrightarrow} \calo_{\cp{2k+1}}(1)^{\oplus c}. 
\end{equation}

Let $\wp:=\{y_k=0\}\subset\cp{2k+1}$ be a hyperplane, and note that the restriction $\tau|_{\wp}:\wp\to\hp{k}$ of the twistor map is surjective and provides a bijection between the affine subsets 
$$ \wp\supset\{x_k\ne0\}=\C^{2k} ~~ \mapsto ~~  \HH^k=\{q_k\ne0\}\subset\hp{k} $$
$$ (x_0,y_0,\cdots,x_{k-1},y_{k-1}) \mapsto (x_0+jy_0,\cdots,x_{k-1}+jy_{k-1}). $$
So lifting a quaternionic instanton connection on $\hp{k}$ to a mathematical instanton bundle on $\cp{2k+1}$ and restricting it to the hyperplane $\wp$ provides an injective map from the moduli space of quaternionic instantons on $\hp{k}$ to the moduli space of rank $2k$ holomorphic bundles on $\wp$ that arise as cohomology of a linear monad similar to the one in display \eqref{monad2}. For the case $k=1$, Donaldson used the so-called \textit{ADHM construction} to show that this map is actually and isomorphism \cite{D}; however, it is unknown to the authors whether the same is true for $k>1$. However, this observation can be regarded as a motivation to consider holomorphic bundles over even dimensional complex projective spaces which arise as the cohomology of a linear monad, see Definition \ref{defn:instantons} below.

%%%%%%%%%%%%%%%%%%%%%%%%%%%%%%%%%%%%%%%%%%%%%%%%%%%%%%%%%%%%%%%%%%%
%%%%%%%%%%%%%%%%%%%%%%%%%%%%%%%%%%%%%%%%%%%%%%%%%%%%%%%%%%%%%%%%%%%

\section{Instanton and perverse instanton sheaves}\label{section:perverse}

From this point onwards, we will shift to denoting sheaves by capital roman letters, while $\mathbb{P}^n$ means $\mathbb{CP}^n$, as it is more usual in algebraic geometry.

We start by recalling the characterization of the family of sheaves on projective spaces that can be represented as the cohomology of a monad as in display \eqref{monad2}.

\begin{theorem}{\label{thm:monadic presentation}}
A torsion free sheaf $E$ on $\pn$ is the cohomology of a monad of the form
$$ U\otimes\opn(-1) \stackrel{\alpha}{\lra} V\otimes\opn \stackrel{\beta}{\lra} W\otimes\opn(1) , $$
where $U$, $V$ and $W$ are vector spaces, if and only if
\begin{enumerate}
\item[(1)] $H^0(E(-1)) = H^n(E(-n)) = 0$ for $n\ge2$;
\item[(2)] $H^1(E(-2)) = H^{n-1}(E(1-n)) = 0$ for $n\ge3$;
\item[(3)] $H^p(E(k)) = 0$ for every $k$ and $2\le p\le n-2$, when $n\ge4$.
\end{enumerate}
\end{theorem}

The proof is given in \cite[Theorem 3]{J-inst}; the \textit{if} part is an application the Beilinson spectral sequence, after some further cohomological vanishings are established; the converse claim is an easy calculation with long exact sequences.

Since the locally free sheaves on odd and on even dimensional complex projective spaces arising from quaternionic instantons on quaternionic projective spaces have $c_1=0$, the following definition definition was proposed in \cite{J-inst}.

\begin{definition} \label{defn:instantons}
An instanton sheaf on $\pn$ is a torsion free sheaf $E$ with $c_1(E)=0$ satisfying the following cohomological conditions
\begin{enumerate}
\item[(1)] $H^0(E(-1)) = H^n(E(-n)) = 0$ for $n\ge2$;
\item[(2)] $H^1(E(-2)) = H^{n-1}(E(1-n)) = 0$ for $n\ge3$;
\item[(3)] $H^p(E(k)) = 0$ for every $k$ and $2\le p\le n-2$, when $n\ge4$.
\end{enumerate}
The number $c:=h^1(E(-1))=c_2(E)$ is called the \textit{charge} of $E$; this is also often called the \textit{quantum number} of $E$. The trivial bundle $\opn^{\oplus r}$ is regarded as an instanton sheaf of charge 0.
\end{definition}

Mathematical instanton bundles, as defined in the previous section, are simply locally free instanton sheaves of rank $2k$ on $\mathbb{P}^{2k+1}$. Therefore, the previous definition generalizes the notion of mathematical instanton bundles to include possibly non locally free sheaves of arbitrary rank on projective spaces of any dimension.

The first observation is that instanton sheaves of rank $r$ less than $n-1$ on $\pn$ are necessarily trivial, i.e. $\opn^{\oplus r}$ \cite[Corollary 6]{J-inst}. Moreover, notice that we do not impose any condition on the stability of $E$; in fact, letting $E$ be a non trivial instanton sheaf of rank $r$ on $\pn$, one can show that:
\begin{enumerate}
\item if $E$ is reflexive and $r\ge n-1$, then $E$ is $\mu$-semistable;
\item if $E$ is locally free and $r\ge 2n-1$, then $E$ is $\mu$-semistable.
\item when $r=2$ and $n=3$, then $E$ is Gieseker stable.
\end{enumerate}
The first two claims are contained in \cite[Theorem 15]{J-inst}, while the third was established in \cite[Theorem 4]{JMaT}. 

Thus, in general, it is not in principle clear how to construct a moduli space of instanton sheaves of arbitrary rank and charge. This issue is addressed in \cite{HLa,HJM} using the ADHM construction of framed instanton bundles on $\pn$, and more recently in \cite{JS} using representations of quivers. Let us comment on both approaches.

\subsection{The ADHM construction and perverse instanton sheaves}

Let $V$ and $W$ be vector spaces of dimension $c$ and $r$, respectively, and consider matrices
$$ A,B \in \End(V)\otimes H^0(\op{n-2}(1)) $$
$$ I \in\Hom(W,V)\otimes H^0(\op{n-2}(1)) ~~,~~ J \in\Hom(V,W)\otimes H^0(\op{n-2}(1)) $$
where $n\ge2$; these are the so-called \textit{ADHM matrices}. Let $\mathbb{X}_n(r,c)$ denote the set of all ADHM matrices as above satisfying the \textit{ADHM equation}:
$$ \mathcal{X}_n(r,c) := \Big\{ (A,B,I,J) ~\Big|~ [A,B]+IJ=0 \Big\} $$
% \Big( \End(V)^{\oplus2} \oplus \Hom(W,V) \oplus \Hom(V,W) \Big) \otimes H^0(\op{n-2}(1)).
The group $GL(V)$ acts on $\mathcal{X}_n(r,c)$ as follows
$$ g\cdot(A,B,I,J) = (gAg^{-1},gBg^{-1},gI,Jg^{-1}). $$
One can then consider the GIT quotient
$$ \mathcal{F}_n(r,c) := \mathcal{X}_n(r,c)/\!\!/ GL(V) ; $$
a quadruple $(A,B,I,J)\in\mathcal{X}_n(r,c))$ is GIT stable if there is no proper subspace $S\subset V$ for which the inclusions $A(S),B(S),I(W)\subset S\otimes H^0(\op{n-2}(1))$, see \cite[Section 2.3]{HJM} and \cite{HLa} for the case $n=3$.

Each point $(A,B,I,J)\in\mathbb{X}_n(r,c)$ can be used to construct a complex of sheaves on $\pn$ as follows, where $X=(A,B,I,J)$:
\begin{equation} \label{eq:fpi}
E^\bullet_{X} ~:~ V\otimes\opn(-1) \stackrel{\alpha}{\longrightarrow} \big( V\oplus V\oplus W\big) \stackrel{\beta}{\longrightarrow} V\otimes\opn(1)
\end{equation}
where the maps $\alpha$ and $\beta$ are given by
$$ \alpha = \left( \begin{array}{c}
A + x\mathbf{1}_V \\
B + y\mathbf{1}_V \\
J
\end{array} \right) ~~ ~~
\beta = \left( \begin{array}{lcr}
-B - y\mathbf{1}_V &
A + x\mathbf{1}_V &
I
\end{array} \right) . $$
To be more clear, let $[z_0:\cdots:z_{n-2}]$ be homogeneous coordinates on $\mathbb{P}^{n-2}$ and $[z_0:\cdots:z_{n-2}:x:y]$ denote homogeneous coordinates on $\pn$. Then the first line of the morphism $\alpha$ can be written in the following way
$$ A_0z_0 + \cdots + A_{n-2}z_{n-2} + x\mathbf{1}_V, $$
where $\mathbf{1}_V$ denotes the identity in $\End(V)$; the other entries of $\alpha$ and $\beta$ can be interpreted in a similar way.

Notice that $\beta\alpha=0$ precisely because the ADHM equation holds. Moreover, $\alpha$ is injective, while $\coker\beta$ has codimension at least 2, since $\beta$ is surjective along the line $\ell=\{z_0=\cdots=z_{n-2}=0\}$; in fact, $\beta$ is surjective if and only if for each $p\in\p{n-2}$, there is no proper $S\subset V$ for which the inclusions $A(p)(S),B(p)(S),I(p)(W)\subset S$, a condition that implies the GIT stability of $X=(A,B,I,J))$.

Finally, two complexes $E_X$ and $E_{X'}$ are isomorphic if and only if $X'=g\cdot X$.

These observations motivates the following definition.

\begin{definition} \label{defn:perverse}
A \textit{perverse instanton sheaf} is an object $E\in D^b(\pn)$ quasi-isomorphic to a complex of the form
$$ \opn(-1)^{\oplus c} \longrightarrow \opn^{\oplus r+2c} \longrightarrow \opn(1)^{\oplus c} $$
satisfying the following conditions
\begin{itemize}
\item $\calh^{p}(E)=0$ for $p\ne0,1$;
\item $\calh^0(E)$ is a torsion free sheaf;
\item $\calh^1(E)$ is a torsion sheaf of codimension at least 2.
\end{itemize}
Note that $r=\rk\big(\calh^0(E)\big)$, and this is called the rank of $E$; a \textit{rank 0 instanton sheaf} is just a perverse instanton sheaf of rank 0. The integer $c=\ch_2(E)$ is called the \textit{charge} of $E$.

In addition, a \textit{framing} on $E$ is a choice of isomorphism $\varphi:\calh^0(E)|_{\ell}\stackrel{\sim}{\to}\calo_{\ell}^{\oplus r}$. The pair $(E,\varphi)$ is called a\textit{ framed perverse instanton sheaf}.
\end{definition}

Therefore, for any $X\in\mathcal{X}_n(r,c)$ the complex $E_X$ presented in display \eqref{eq:fpi} is a perverse instanton sheaf of rank $r$ and charge $c$, equipped with a framing $\varphi:\calh^0(E)|_{\ell}\stackrel{\sim}{\to}W\otimes\calo_{\ell}$. The GIT quotient $\calf_n(r,c)$ can then be interpreted as the moduli space of \textit{GIT stable} framed perverse instanton sheaves of rank $r$ and charge $c$.

Every instanton sheaf in the sense of Definition \ref{defn:instantons} is a perverse instanton sheaf with $\calh^1(E)=0$; and every perverse instanton sheaf with $\calh^1(E)=0$ is just an instanton sheaf. Therefore, every framed instanton sheaf is semistable, in the sense that it corresponds to a GIT semistable ADHM datum.

The main example of a perverse instanton sheaf that is not a sheaf is the derived dual of a non locally free instanton sheaf. In general, the $0^{\rm th}$-cohomology of a perverse instanton sheaf may not be an instanton sheaf.

\subsection{Instantons as representations of a quiver}\label{sec:repQuiver}

The information contained in a linear monad of the form
\begin{equation} \label{monad3}
V\otimes\opn(-1) \stackrel{\alpha}{\longrightarrow} W\otimes\opn \stackrel{\beta}{\longrightarrow} U\otimes\opn(1)
\end{equation}
can be neatly packaged as a representation of the quiver
\begin{equation} \label{Q}
\mathbf{Q} := \left\{\begin{tikzcd}
\underset{-1}{\bullet} \arrow[rr,bend left,"\alpha_0"] \arrow[rr,bend right,swap,"\alpha_n"] &\vdots & \underset{0}{\bullet} \arrow[rr,bend left,"\beta_{0}"] \arrow[rr,bend right,swap,"\beta_n"] &\vdots &\underset{1}{\bullet}
\end{tikzcd}\right\}
%\xymatrix{
%\underset{-1}{\bullet} \ar@< 3pt> [r]^{\alpha_0} \ar@<-9pt> [r]^\vdots_{\alpha_{n}}	 &
%\underset{0}{\bullet}  \ar@< 3pt> [r]^{\beta_0} \ar  @<-9pt> [r]^\vdots_{\beta_{n}}	 &
%\underset{1}{\bullet} }\Biggr\}
\end{equation}
with $n+1$ arrows between each vertex, satisfying the relations
\begin{equation}\label{eq:relations}
\beta_j\alpha_i+\beta_i\alpha_j=0 ~~{\rm with}~~ 0\le i,j \le n.
\end{equation}
Indeed, we place the vector spaces $V$, $W$ and $U$ on the vertices $-1$, $0$, and $1$, respectively. Set $[x_0:\dots:x_n]$ as homogeneous coordinates on $\pn$; the morphisms $\alpha$ and $\beta$ can then be written as follows
$$ \alpha = \sum_{i=0}^n A_ix_i ~~{\rm and} ~~ \beta = \sum_{i=0}^n B_ix_i $$
where $\alpha_i\in\Hom(V,W)$ and $\beta_i\in\Hom(W,U)$. To complete the representation of the quiver $\mathbf{Q}$, we attach the matrices $A_i$ to the arrows $\alpha_i$, while the matrices $B_i$ are attached to the matrices $\beta_i$. Finally, the fact that $\beta\alpha=0$ implies that the relations in display \eqref{eq:relations} are satisfied. The injectivity of $\alpha$ and surjectivity of $\beta$ impose further (open) conditions on the set of representations of $\mathbf{Q}$ that come from linear monads; further details and generalizations can be found in \cite{JP} and in \cite{JS}.

Turning back our attention to the moduli space of instanton sheaves, we observe that the dimension vector of a representation of $\mathbf{Q}$ associated to the monad for an instanton sheaf of rank $r$ and charge $c$ is given by $(c,r+2c,c)$. One can then consider the King moduli space $\calr_\theta(r,c)$ of $\theta$-semistable representations of $\mathbf{Q}$ with fixed dimension vector $(c,r+2c,c)$; here, the stability parameter $\theta$ is given by
$$ \theta = \big( \alpha , -(\alpha+\gamma)\frac{c}{r+2c} , \gamma \big) ~~{\rm with}~~ \alpha,\gamma\in\R.$$
It is not difficult to see that $\calr_\theta(r,c)$ is empty whenever $\alpha>0$ and $\gamma<0$, see \cite[Lemma 7]{JS}.

The case of rank 2 instanton sheaves was studied in detail in \cite{JS}. One can show that if $E$ is an instanton sheaf, then there is $\theta$ such that the corresponding representation of $\mathbf{Q}$ is $\theta$-stable \cite[Proposition 8]{JS}; moreover, there is a wall in the $\alpha\gamma$-plane that destabilizes every representation corresponding to a non locally free instanton sheaf, so that perverse instanton sheaves do correspond to certain $\theta$-stable representation of $\mathbf{Q}$.

In summary, there are at least two ways to construct reasonable (i.e., projective) moduli spaces of instantons sheaves of arbitrary rank and charge: via the ADHM construction, or via moduli spaces of representations of quivers. However, both constructions include more complicated objects in the derived category of sheaves, like the perverse instantons sheaves considered in Definition \ref{defn:perverse}.

%%%%%%%%%%%%%%%%%%%%%%%%%%%%%%%%%%%%%%%%%%%%%%%%%%%%%%%%%%%%%%%%%%%
%%%%%%%%%%%%%%%%%%%%%%%%%%%%%%%%%%%%%%%%%%%%%%%%%%%%%%%%%%%%%%%%%%%

\section{Instanton sheaves on Fano threefolds}\label{section:Fano}

In section \ref{section:perverse}, we saw how to generalize the ``classical" notion of instanton extending the definition to torsion free sheaves of arbitrary rank and even to objects belonging to the derived category $D^b(\pn)$. Another possible direction is to construct instantons on projective varieties beside projective spaces. For the particular cases of Fano threefolds of Picard rank one, this was done in \cite{F,K,CJ}.

\subsection{Instanton bundles on Fano threefolds}
The rank 2 locally free instantons on the Fano threefolds of Picard rank one are the main subject  \cite{F}. We summarize here the main results of Faenzi's work.
To begin with we consider a Fano threefold $X$ of Picard rank one and we denote by $H_X$ the ample generator of $\Pic(X)\simeq\mathbb{Z}$. We then write the anticanonical class $K_X$ as $K_X=-i_X H_X$. We have that $i_X$ is a positive (since $X$ is Fano) integer, referred to as the \textit{index} of $X$, that takes values in $i_X\in \{1,2,3,4\}$. 
We set $i_X=2q_X+e_X$, where $q_X$ and $e_X$ are integers such that $q_X\ge 0$ and $0\leq e_X\leq 1$. The definition of instanton presented in \cite{F} is the following:
\begin{definition}\label{defn:inst-F}
An instanton bundle on $X$ is a rank 2 stable bundle with $c_1=-e_X$, and such that 
$$ E\simeq E^*(-e_X), \hspace{3mm} H^1(E(-q_X))=0.$$
\end{definition}
\noindent Note that the \textit{instantonic condition} $H^1(E(-q_X))=0$ is the analogue of the vanishing $H ^1(E(-2))$ holding for instantons on $\p3$.
Using the Serre's correspondence and the stability assumption it can be shown that the  instantons satisfy the following cohomological vanishing: 
\begin{lemma}\label{lem:vanish_bundle}
If $E$ is a rank 2 instanton bundle on $X$ we then have:
\begin{equation}\label{orth-bundle}
    H^i(E(-q_X))=0\ \text{and}\  \Ext^i(E,\calo_X(-q_X-e_X))=0,\ \text{for all}\ i;
\end{equation}
and $H^1(E(-q_X-t))=0, \: H^2(E(-q_X+t))=0$,  for all  $t\ge 0$.
\end{lemma}

\subsection{Non-emptiness of moduli spaces of instantons }
%\textcolor{blue}{summary of theorem A
%add a remark on the use of non locally free inst to prove existence of components.}
%dimension of components of moduli spaces-introduction of elementary transformation.
The first main result of \cite{F}, concerns the non-emptiness of the moduli space $\cali(n)$ of instanton bundles of charge $n$ on all Fano threefolds $X$ of Picard rank one and index $i_X>1$ and on non-hyperelliptic Fano threefolds of index one (which means that $-K_X$ is very ample) containing a line $\ell\subset X$ with normal bundle $\calo_{\ell}\oplus \calo_{\ell}(-1)$.
\begin{theorem}
The moduli space $\cali(n)$ has a generically smooth irreducible component whose dimension is the number $\delta$ below: 

\begin{center}
    \begin{tabular}{@{}l|cccc@{}}
\toprule
   $i_X$  & $4$ & $3$ & $2$ & $1$ \\ \hline
$\delta$ &  $8n-3$  & $6n-6$ & $4n-3$ & $2n-g_X-2$\\

 \bottomrule
\end{tabular}
\end{center}
and $\cali(n)$ is empty when $i_X=2$ and $n=1$, and when $i_X=1$ and $2n<g_X+2$.

\end{theorem}
The integer $g_X$ appearing in the statement of the theorem is the \textit{genus} of $X$; we recall that this parameter, defined on Fano threefolds of index one, is the genus of a general codimension 2 plane section of $X$.
\begin{remark}\label{rmk:et}\textbf{Existence of instantons on Fano threefolds of index 2.}
%\newline
In the particular case of Fano threefolds of index 2, the proof of the non-emptiness of the moduli spaces $\cali(n)$, for $n> 2$ relies on the construction of divisors in $\overline{\cali(n)}$ parameterizing non-locally free sheaves $E$ that still satisfy the cohomological vanishing $H^{\bullet}(E(-1))=0$ (note that all these conditions are, by semicontinuity, open).
To be more precise, these sheaves $E$ are \textit{elementary transformations} of instanton bundles $F$ of charge $n-1$ along structure sheaves $\calo_{\ell}$ of lines ${\ell}\subset X$. This means that the sheaves $E$ fit into short exact sequences of the form:
\begin{equation}\label{singular-line}
    0\lra E \lra F\lra \calo_{\ell}\lra 0
    \end{equation}
(from which we learn in particular that $F\simeq E^{**}$ and $\Sing(F)=\ell$).
The non-emptiness of $\cali(n)$ can then be proved applying an induction argument. To begin with we show the existence of a generically smooth irreducible component of $\cali(2)$ (this is done applying the well-known Serre's correspondence relating locally complete intersection curves on $X$ with rank 2 bundles). The induction step consists then in showing that for a general pair $(F, \ell)$ with $[F]\in\cali(n-1)$ and $\ell\subset X$ a line on $X$, the general deformation of a sheaf $E$ fitting into a short exact sequence of the form (\ref{singular-line}) is an instanton bundle of charge $n$. 
% i should find a way to explain why i wrote all this
This procedure suggested that, more generally, the investigation on families of rank 2 non-locally free sheaves $E$ with $H^i(E(-q_X))=0,$ might contribute to get a better understanding of $\overline{\cali(n)}$. 
\end{remark}
\

\subsection{Monadic representations of instantons}
Faenzi focuses then his attention on instantons defined over Fano threefolds $X$ such that $H ^3(X)=0$. On these threefolds, the instanton bundles share another common feature with instantons on the projective space: they can still be represented as cohomology of monads. This property is remarkable for the following reasons: in the first place it provides us with a ``recipe" to construct instantons, in the second place it allows us to construct their moduli as GIT quotients. As it turns out, the condition $H^3(X)=0$ is indeed equivalent to the fact that $X$ admits a full strong exceptional collection; this allows to prove analogues of the Beilinson's theorem on $\pn$. More specifically, on a Fano threefold $X$ such that $H^3(X)=0$, there exist vector bundles $\cale_i, \ i=0,\dots, 3$, satisfying
$$ \cale_0\simeq\calo_X(-q_X-e_X), \hspace{2mm} \cale_3^*(-e_X)\simeq\cale_1, \hspace{2mm}
\cale_2^*(-e_X)\simeq\cale_2$$
and such that $D^b(X)=\langle \cale_0,\cale_1,\cale_2,\cale_3\rangle$.
Denoting then by $\langle\calf_0,\calf_1,\calf_2,\calf_3\rangle$ the dual collection, we have that each coherent sheaf $E$ on $X$ is the cohomology of a complex $\calc^{\bullet}_{E}$
with $\calc^{j}_E=\oplus_i H^i(F\otimes \calf_{j-i+3})\otimes \cale_{j-i+3}$ 
where the index $i$ runs between $\max\{0,j\}$ and $\min\{3,j+3\}$.
In the particular case in which $E$ is an instanton bundle, the complex $\calc^{\bullet}_E$ is a monad whose terms can be described, in further detail, as follows.
For an integer $n$ let us fix vector spaces $I$ and $W$ whose dimensions are subjected to the following constraints:
\begin{center}\label{charge-monads}
    \begin{tabular}{@{}lccc@{}}
\toprule
   $i_X$& $n$& $\dim(I)$ & $\dim(W)$ \\ \midrule
$4$ &  $n\ge 1$  & $n$ & $2n+2$ \\
$3$ & $n\ge 2$   & $n-1$ & $n$ \\
$2$ & $n \ge 2$   & $n$ & $4n+2$\\
$1$ & $n \ge 8$   & $n-7$ & $3n-20$ \\
 \bottomrule
\end{tabular}
\end{center}
and let us denote by $U$ the vector space $U:=\Hom(\cale_2,\cale_3).$
We fix then an isomorphism $D:W\to W^*$ such that $D^t=(-1)^{e_X+1}D$ and we consider the locally closed subvariety $\overset{\circ}{\cald_{X,n}}$ of the vector space $\Hom(W^*\otimes \cale_2, I\otimes \cale_3)\simeq I\otimes W\otimes U$ defined as
$$\overset{\circ}{\cald_{X,n}}:=\{ A\in \Hom(W^*\otimes \cale_2, I\otimes \cale_3)\mid A\:D\:A^t=0 \ \text{and A is surjective}\}.$$
Finally, we write $G(W,D)$ for the symplectic group $Sp(W,D)$, or for the orthogonal group $O(W,D)$, depending on whether $e_X=0,1$; the group $G_n:=GL(I)\times G(W,D)$ acts then on $\overset{\circ}{\cald_{X,n}}$ via $(\zeta,\eta)\cdot A= (\zeta A \eta^t)$.
The second main result of \cite{F} is the following:
\begin{theorem}{\label{thm:monadFano}}
Let $X$ be a smooth Fano threefold of Picard rank one and such that $H^3(X)=0$. Let $I,\: W,\: D,\: \cale_i$ as above. Then an instanton $E$ of charge $n$ on $X$ is the cohomology of a monad of the form
$$ I^*\otimes \cale_1\xrightarrow{DA^t}W^*\otimes \cale_2\xrightarrow{A} I\otimes \cale_3,$$
and conversely the cohomology of such a monad is an instanton of charge $n$.
The moduli space $\cali(n)$ of instanton bundles of charge $n$ is isomorphic to the geometric quotient $\overset{\circ}{\cald_{X,n}}/G_n$.
\end{theorem}
\subsection{Instanton sheaves on Fano threefolds.}
The main properties of rank 2 instanton bundles on Fano threefolds, illustrated in \cite{F}, appear as ``natural generalizations" of the properties of mathematical instantons on the projective space.
We might then wonder if something similar still happens if we extend our study to sheaves of arbitrary rank and that are not necessarily locally free; in other words we might try to adapt to the Fano threefolds besides $\p3$ the approach adopted by Jardim in \cite{J-inst}. 
This issue had been dealt in \cite{CJ} where the following definition of \textit{instanton sheaf} is presented (the notations adopted are the one we introduced in the previous section):
\begin{definition}\label{defn:inst-CJ}
Let $X$ be a Fano threefold of Picard rank one and index \mbox{$i_X=2q_X+e_X$}, where $q_X$, $e_X$ are integers such that $q_X\ge0$ and $0\le e_X\le 1$.
An \textit{instanton sheaf} $E$ on $X$ is a torsion free $\mu$-semistable sheaf with first Chern class $c_1(E)=-e_X$ and such that:
\begin{equation}
    H^1(E(-q_X))=H^2(E(-q_X))=0.
\end{equation}
The \textit{charge} of $E$ is defined to be $c_2(E)$. 
\end{definition}
\begin{remark}{\label{rmk:mu-unstable}}
Note that this definition appears to be more restrictive than the one adopted in \cite{J-inst} since this latter does not necessarily implies $\mu$-semistability (see e.g. \cite[Example 3]{J-inst}).
\end{remark}
Moving to this more general setting, some of the cohomological characterizations of instantons presented in Lemma \ref{lem:vanish_bundle} still hold:
\begin{lemma}\label{lem:vanish_bundle2}
Let $E$ be an instanton sheaf. Then:
$$H^i(E(-q_X))=\Ext^i(E,\calo_X(-q_X-e_X))=0 \ \ \text{for all}\ \ i.$$
\end{lemma}
A conspicuous part of \cite{CJ} is devoted to the study of the non-locally free instanton sheaves. An efficient way to produce sheaves of such a kind is performing \textit{elementary transformations} of instantons along rank 0 instantons.
\begin{definition}\label{defn:rank-0}
A rank 0 instanton sheaf on $X$ is a 1-dimensional sheaf $T$ satisfying
$H^i(T(-q_X))=0, \ i=0,1$.
\end{definition}
\begin{remark}\label{rmk:pure-dim}
The vanishing of $H^0(T(-q_X))$ implies that $H^0(T(-n))=0$ for $n\gg 0$.
Accordingly a rank 0 instanton must have pure dimension 1 (that is to say it admits no zero-dimensional subsheaf).
\end{remark}
The notion of elementary transformation had already been introduced in Remark \ref{rmk:et}: we say that $E$ is the elementary transformation of an instanton $F$ along a rank 0 instanton $T$ if $E$ fits into a short exact sequence of the form:
\begin{equation}\label{eq:et}
    0\lra E\lra F\lra T\lra 0.
\end{equation}
From this short exact sequence we can easily verify that the sheaf $E$ is indeed an instanton that moreover satisfies $E^{**}\simeq F^{**}$; in particular if ever $F$ is reflexive, $F\simeq E^{**}$.
Notice therefore that the non-locally free sheaves $E$ constructed in Remark \ref{rmk:et} and belonging to the boundary $\partial\overline{\cali(n)}$ are instanton sheaves: they are indeed obtained performing elementary transformation of rank 2 instanton bundles $F$ (so that, in particular, $F\simeq E^{**}$) along structure sheaves of lines $\calo_{\ell}$ (these latter are rank 0 instantons on Fano varieties of index 2 since $H^i(\calo_{\ell}(-1))=0$ for $i=0,1$). 
Via the technique described above, we can thus construct families of non-reflexive instantons with  1-dimensional singular locus.
The main properties of non-reflexive instantons are summarized in the following proposition:
%Nevertheless it is in general not true that all non-reflexive instantons are obtained in such a way.
\begin{proposition}\label{prop:non-relf}
Let $E$ be a non-reflexive instanton sheaf of rank $r>0$. Then the following hold:
\begin{itemize}
    \item $T_E:=E^{**}/E$ has pure dimension one;
    \item $E$ has homological dimension one;
    \item $E^{**}$ is an instanton if and only if $T_E$ is a rank 0 instanton. 
\end{itemize}
\end{proposition}
Summing up, we can always construct families of non-reflexive instantons via elementary transformation of reflexive instantons along rank 0 instantons but, from Proposition \ref{prop:non-relf} we learn that, in general, not all non-reflexive instanton are obtained in this way.
Nevertheless this last assertion holds true if we restrict to the rank two case. 
\begin{theorem}\label{thm:classification-rk2}
Let $E$ be a rank 2 instanton sheaf. Then $E^{**}$ is an instanton bundle and $T_E:=E^{**}/E$ is a rank 0 instanton whenever $T_E\ne 0$.
\end{theorem}
\begin{remark}\label{rmk:sing-rk2 }
From \ref{rmk:pure-dim} and Theorem \ref{thm:classification-rk2} we learn
the following: a rank 2 instanton $E$ either does not present singularities or it has purely one dimensional singular locus $\Sing(E)=\Supp(E^{**}/E)$.
\end{remark}
%question: mention that we can copy daniele and prove the existence of instantons? (thm 29 of \cite{CJ})
% E^{**}\simeq F^{**} hence F slope-ss iff E slope-ss
%perhaps write this in terms of derived categories
In the rank one case the investigation of non-reflexive instantons even lead to a complete classification of the rank one instantons.
\begin{proposition}\label{prop:rank-one}
Let $L$ be a rank 1 instanton sheaf of charge $n$ on a Fano threefold $X$ with Picard rank one. The following hold:
\begin{itemize}
    \item if $i_X=3,4$ then $n=0$ and $L\simeq \calo_X(-e_X)$;
    \item if $i_X=1,2$, we have $L\simeq \calo_X(-e_X)$ whenever $n=0$ whilst for $n>0$, $L$ always fits in a short exact sequence of the form:
    $$0\rightarrow L\rightarrow L'\rightarrow \calo_{\ell}(-e_X)\rightarrow 0$$
    for a line $\ell\subset X$ and a rank one instanton $L'$ of charge $n-1$.
\end{itemize}
\end{proposition}
\begin{comment}
Each rank one instanton $L$ of charge $n$ on a Fano threefold $X$ of index $i_X=1$ or $2$ is therefore always isomoprhic to $\cali_C(-e_X)$ for $C$ a l.c.m. curve $C$ of degree $n$ satisfying $H^i(\calo_C(-1))=0, \ i=0,1$. 
Curves of such a kind cab be constructed "inductively" from extensions:
$$ 0\rightarrow \calo_l\rightarrow calo_C\rightarrow \calo_{C'}\rightarrow 0$$
for $C'$ a l.c.m. curve of degree $n-1$ and satisfying the same cohomological conditions $H^i(\calo_{C'}(-1))=0$. 
\end{comment}
%section on elementary transformation
\subsection{Instanton bundles on Fano threefolds of index 2}

The rank 2 instanton bundles on a Fano threefold $X$ of index 2 are also the main subject of Kuznetsov's work \cite{K}. In the article the author's attention is mainly drawn to %gives a particular emphasi
the behavior of these bundles seen as objects in the derived category  $D^b(X)$ of $X$. 
The definition of instanton provided by Kuznetsov is the following:
\begin{definition}\label{defn:kutz}
Let $X$ be a Fano threefold of index 2. An \textit{instanton of charge n} is a stable vector bundle $E$ of rank 2 with $c_1(E)=0, \ c_2(E)=n$ and such that $H^1(E(-1))=0.$

\end{definition}
We notice therefore that on $X$, the definition of instanton adopted by Kuznetsov coincides the one presented by Faenzi.
% the fact that they are even instanton is obvious since the index is even
Kuznetsov's investigation of instanton bundles on $X$ starts with the computation of their cohomology table.
\begin{lemma}\label{lem:cohom-inst}
Let $E$ be an instanton bundle of charge $n$ on a Fano threefold $X$ of index 2. Then the cohomology table of $E$ has the following shape:
\begin{center}
    \begin{tabular}{@{}l|lcccccc@{}}
\toprule
   $t$& $\cdots $& $-3$ & $-2$ & $-1$ & $0$ & $1$ & $\cdots $ \\ \hline
   &&&&&&& \\
$h^3(E(t))$ &  $\cdots$  & $\ast$ & $0$ & $0$ & $0$ & $0$ & $\cdots $ \\
&&&&&&& \\
$h^2(E(t))$ & $\cdots$   & $\ast$ & $n-2$ & $0$ & $0$ & $0$ &$\cdots$     \\
&&&&&&& \\
$h^1(E(t))$ & $\cdots$   & $0$ & $0$ & $0$ & $n-2$ & $\ast$ &$\cdots$ \\
&&&&&&& \\
$h^0(E(t))$ & $\cdots$   & $0$ & $0$ & $0$ & $0$ &$\ast$ &$\cdots$ \\
 \bottomrule
\end{tabular}
\end{center}
 \end{lemma}
As an immediate corollary we also get the following:
\begin{corollary}
The charge of an instanton is greater or equal than $2$.
\end{corollary}

\begin{remark}
\noindent
\begin{itemize}
\item By \cite[Theorem 3.1]{F}, we know that instantons of charge 2 indeed exist so that 2 is actually the minimal value of the charge of an instanton bundle on a Fano threefold of index 2. However, $\cali_\ell\oplus\calo_X$ is an example of a non locally free instanton sheaf of charge 1.
\item By the table displayed in Lemma  \ref{lem:cohom-inst}, we see that since $H^0(E)=0$, the Gieseker stability of an instanton bundle $E$ actually coincides with its slope-stability.
%(features of strictly slope-semistable instantons of rank 2 were presented in \cite{CJ}).
\item Because of the stability assumption required in Definition \ref{defn:kutz}, we see that a vector bundle that is an instanton in the sense of Kuznetsov (or equivalently of Faenzi) is clearly an instanton in the sense of Definition \ref{defn:inst-CJ}. Nevertheless the converse implication does not hold: it is indeed shown in \cite{CJ} that Fano threefolds of index 2 admit strictly $\mu$-semistable rank 2 vector bundles $E$ with $ch(E)=(2,0,-n,0)$ and $H^i(E(-1))=0$, for all $i$. These bundles are therefore instantons according to Definition \ref{defn:inst-CJ} but not in the sense of Definitions \ref{defn:kutz} and \ref{defn:inst-F}. In loc.cit. it is shown that actually the Fano threefolds of index 2 indeed are the only ones carrying families of strictly $\mu$-semistable rank 2 instanton bundles of charge $n>0$ and that moreover each instanton $E$ of such a kind fits into a short exact sequence:
$$ 0\lra \calo_X\lra E\lra L\lra 0 $$
with $L$ a rank one instanton of charge $n$.
\end{itemize}

\end{remark}

\subsection{The acyclic extension of instantons} As we mentioned before, one of Kuznetsov's main aims is to describe the properties of instantons using the language of derived categories.
We recall that for a Fano threefold $X$ of index 2, the collection of line bundles $\calo_X, \ \calo_X(1)$ is exceptional; accordingly we obtain the following semiorthogonal decompositon of the  derived category $D^b(X)$:
\begin{equation*}
   D^b(X)= \langle \calb_X, \calo_X,\calo_X(1)\rangle, \hspace{2mm} \calb_X:=\langle \calo_X, \calo_X(1)\rangle ^{\perp}.
\end{equation*}
Starting from an instanton $E$, as $E\in \langle\calo_X(1)\rangle^{\perp}$ (this is due to Lemma \ref{lem:cohom-inst}), we can construct an object $\tilde{E}\in \calb_X$ performing a \textit{left mutation through $\calo_X$}. We recall that the left mutation through $\calo_X$ is the functor $\mathbb{L}_{\calo}:D^b(X)\rightarrow \langle\calo_X\rangle^{\perp}$ sending an object $F\in D^b(X)$ to the cone of the evaluation morphism $\Ext^{\bullet}(\calo_X,F)\otimes \calo_X\to F$.
Since for an instanton bundle $E$, the complex $\Ext^{\bullet}(\calo_X,E)\otimes \calo_X$ is concentrated in degree $-1$ (once again, this is due to Lemma \ref{lem:cohom-inst}), $\tilde{E}:=\mathbb{L}_{\calo}(E)$ is actually a sheaf object that fits into a short exact sequence:
$$ 
0\lra E\lra\tilde{E}\lra \calo_X^{n-2}\lra 0.
$$
The sheaf $\tilde{E}$ is referred to as the \textit{acyclic extension of $E$}.
\begin{lemma}\label{lem:acyclic-inst}
The acyclic extension of an instanton $E$ is a simple slope-semistable vector bundle $\tilde{E}$ with $\ch(\tilde{E})=(n, \: 0, \: -n,\:0)$ and such that $
H^{\bullet}(\tilde{E})=H^{\bullet}(\tilde{E}(-1))=0$. Moreover $h^{0}(\tilde{E})=h^{1}(\tilde{E})=n-2$ and $h^{2}(\tilde{E})=h^{3}(\tilde{E})=0.$
\end{lemma}
%\begin{remark} the acyclic extension of a rank 2 instanton bundle is therefore an instanton in the sense of \cite{CJ}.
Recall now that since an instanton $E$ (in the sense of Definition \ref{defn:kutz}) has rank 2 and first Chern class 0, it is \textit{self-dual}; this property implies in particular a ``generalized self-duality" of its acyclic extension. 
Consider indeed the functor \mbox{$D:D^b(X)\to D^b(X), \ F\mapsto \mathbb{L}_{\calo}(R\inhom(F,\calo_X))$}. It is not difficult to prove that the functor $D$ satisfies the following properties:
\begin{lemma}
\noindent
\begin{itemize}
    \item There exists a natural isomorphism $\delta:D^2\xrightarrow{\sim} id$;
    \item the category $\calb_X$ is preserved by $D$.
\end{itemize}
\end{lemma}
Once we have defined the functor $D$ we can state the self-duality property of acyclic extensions.
\begin{proposition}\label{prop:acyclic-selfdual}
Let $\tilde{E}$ be the acyclic extension of an instanton bundle $E$. Then there exists a skew-symmetric isomorphism $D(\tilde{E})\xrightarrow{\phi} \tilde{E}$, in the sense that it fits into a commutative diagram:
$$
\begin{tikzcd}
 & D^2(\tilde{E}) \arrow{dr}{\delta_{\tilde{E}}} \\
D(\tilde{E}) \arrow{ur}{D(\phi)} \arrow{rr}{-\phi} && \tilde{E}.
\end{tikzcd}
$$
\end{proposition}
As it turns out, an instanton bundle $E$ can be ``reconstructed" from its acyclic extension. Each vector bundle $F$ satisfying the properties of both Proposition \ref{prop:acyclic-selfdual} and Lemma \ref{lem:acyclic-inst} is indeed the acyclic extension of a \textit{unique} instanton bundle. 
\begin{theorem}\label{thm:reconstruction}
Let $F$ be a vector bundle on $X$ with $\ch(F)=(n, \: 0, \: -n,\:0)$ and such that $
H^{\bullet}(F)=H^{\bullet}(F(-1))=0$. Then $h^i(F^*)=0$ for $i>1$ and \mbox{$h^0(F^*)=h^1(F^*)\le n-2$}; if moreover  $h^0(F^*)=n-2$, then there exists a unique instanton bundle $E$ of charge $n$ such that $F\simeq \tilde{E}$.

\end{theorem}

It is worth mentioning that also in this setting, we have that the ideal sheaves of lines $\cali_{\ell}, \ \ell\subset X$, share several common features with instantons, or better to say, with their acyclic extensions. 
It is immediate to prove that $\cali_{\ell}\in\calb_X$; moreover the following holds
\begin{proposition}
Ideal sheaves $\cali_{\ell}$ of lines $\ell\subset X$ are fixed by $D$: $D(\cali_{\ell})\simeq\cali_{\ell}$. Moreover this isomorphism is skew-symmetric (in the sense of Proposition \ref{prop:acyclic-selfdual}).
\end{proposition}

\section{$h$-instanton sheaves on projective varieties}\label{section:h-inst}

More recently, several authors have further extended the notion of instanton bundles to beyond Fano 3-folds. This was, once again, initially motivated by gauge-theory: recall that $\p2$, just like $S^4$ also has the structure of a quarternionic K\"ahler manifold, and its twistor space is the full flag manifold $F(0,1,2)$ of points and lines in $\cp2$; the Atiyah--Ward correspondence provides in this case a correspondence between quaternionic instantons on $\cp2$ and a class of holomorphic bundles on $F(0,1,2)$, and these are again called instanton bundles.

Following \cite{DFlag}, Marchesi, Malaspina and Pons-Llopis provided a mathematical treatment of the instanton bundles on  $F(0,1,2)$ in \cite{MMPL}. This case study further motivated the introduction of the notion of instanton bundles on other Fano threefolds with Picard rank larger than 1, see for instance \cite{AM20,ACG,ACG21,CCGM}. 

In the latest development, Antonelli and Casnati defined a class of sheaves on a projective scheme $X$ with respect to an ample and globally generated line bundle $\mathcal{O}_X(h)$ via certain cohomological vanishing conditions that generalize the example that have been previously studied (projective space, Fano threefolds of Picard rank 1); they show that these sheaves can be constructed via monads, so it is reasonable to call them instanton sheaves on $X$. In this section, we review the definition and constructions of these instanton sheaves following \cite{AC}; let us start with the definition proposed by these authors.

\begin{definition}{\label{defn:h-instanton}}{\cite[Definition 1.3 and Theorem 1.4]{AC}}
Let $X$ be an irreducible projective scheme of dimension $n$ ($n\geq 1$) endowed with an ample and globally generated line bundle $\mathcal{O}_X(h)$. If $\mathcal{E}$ is a coherent sheaf on $X$, $k$ a non–negative integer and $\delta \in \{0, 1\}$, then $\mathcal{E}$ is called an $h$-instanton sheaf if the following assertions hold:
	\begin{enumerate}
	\item $h^0(\mathcal{E}(-h))=h^n(\mathcal{E}((\delta-n)h))=0$;

	\item $h^i(\mathcal{E}(-(i+1)h))=h^{n-1}(\mathcal{E}((\delta-n+i)h))= 0$ if $1\leq i\leq n-2$;

	\item $\delta h^i(\mathcal{E}(-ih))=0$ for $2\leq i\leq n-2$;

	\item $h^1(\mathcal{E}(-h))=h^{n-1}(\mathcal{E}((\delta-n)h))=k$;

	\item $\delta(\chi(\mathcal{E})-(-1)^n\chi(\mathcal{E}(-nh))))=0$.
	\end{enumerate}

\end{definition}

The following chart gives the cohomologies of an $h-$instanton sheaf $\mathcal{E}$ with defect $\delta$ and quantum number $k$ when $n\geq 4$.

\begin{center}
\begin{tabular}{@{}l|ccccccccc@{}}
\toprule
$t$ &$\cdots$& $-n-1$ & $-n$ & $-n+1$ & $\cdots$ & $-2$& $-1$& $0$&$\cdots$\\
\hline
$h^n(\mathcal{E}(t))$&$\cdots$&$\ast$&\diagbox{$\ast$}{$0$}&$0$&$\cdots$&$0$&$0$&$0$&$\cdots$\\
\hline
$h^{n-1}(\mathcal{E}(t))$&$\cdots$&$\ast$&\diagbox{$k$}{$\ast$}&\diagbox{$0$}{$k$}&$\cdots$&$0$&$0$&$0$&$\cdots$\\
\hline
$h^{n-2}(\mathcal{E}(t))$&$\cdots$&$0$&$0$&$0$&$\cdots$&$0$&$0$&$0$&$\cdots$\\
$\vdots$& & & & & & & & & \\
$h^2(\mathcal{E}(t))$&$\cdots$&$0$&$0$&$0$&$\cdots$&$0$&$0$&$0$&$\cdots$\\
%\hline
$h^1(\mathcal{E}(t))$&$\cdots$&$0$&$0$&$0$&$\cdots$&$0$&$k$&$\ast$&$\cdots$\\
$h^0(\mathcal{E}(t))$&$\cdots$ &$0$&$0$&$0$&$\cdots$&$0$&$0$&$\ast$&$\cdots$\\
\bottomrule
\end{tabular}
\end{center}

A cell $\begin{tabular}{|c|}\hline\diagbox{a}{b}\\\hline\end{tabular}$ in the above chart means that the cohomology is equal to $b$ when $\delta=0$, and is equal to $a$ when $\delta=1$.
An $h$-instanton sheaf $\mathcal{E}$ with $\delta=0$ is called an ordinary instanton, and $\mathcal{E}$ is called non-ordinary if $\delta=1$.

\begin{remark}{\label{rmk:h-inst unstable}}
Definition \ref{defn:h-instanton} doesn't require stability for an $h-$instanton sheaf. Even for the case that $X$ is smooth and $\mathcal{E}$ is an $h-$instanton bundle of rank $2$, $\mathcal{E}$ could be strictly $\mu-$semistable or $\mu-$unstable which is similar to the situation for an instanton sheaf on $\pn$ (Remark \ref{rmk:mu-unstable}). We refer to \cite[Proposition 8.4]{AC} for more details.   
\end{remark}

\subsection{$h$-instanton sheaves on projective spaces and projective schemes}{\label{subsection:h-instOnProj}}

When $X\cong \mathbb{P}^n$, we choose the ample line bundle to be $\mathcal{O}_{\mathbb{P}^n}(h):= \mathcal{O}_{\mathbb{P}^n}(1)$. Then the definition of an ordinary $h-$instanton sheaf coincides with Definition \ref{defn:instantons} in section \ref{section:perverse}, meaning that if $\mathcal{E}$ is an ordinary $h$-instanton sheaf with rank $r$ and charge $c$, then $\mathcal{E}$ is an instanton sheaf in the sense of \cite{J-inst}. Thanks to Theorem \ref{thm:monadic presentation}, $\mathcal{E}$ has the following monadic presentation, and conversely, the cohomology of such a presentation is an $h-$instanton sheaf provided that it's torsion free.

$$
0\to \mathcal{O}^{\oplus c}_{\mathbb{P}^n}(-1)\to \mathcal{O}^{\oplus 2r+c}_{\mathbb{P}^n} \to \mathcal{O}^{\oplus c}_{\mathbb{P}^n}(1)\to 0
$$

This property generalizes to non-ordinary $h-$instanton sheaves. It is proved in \cite[Proposition 3.2]{AC} that a non-ordinary $h-$instanton sheaf is the cohomology of the monad given below, and the cohomology of such a monad is a non-ordinary $h-$instanton provided $b_1, b_2\geq \chi(\mathcal{E})$.

$$
0\to \mathcal{M}^{-1}\to \mathcal{M}^{0}\to \mathcal{M}^{1}\to 0,
$$
in which

\begin{align*}
\mathcal{M}^{-1}&:=\mathcal{O}_{\mathbb{P}^n}(-1)^{\oplus b_1-\chi(\mathcal{E})}\\
\mathcal{M}^{0}&:=\begin{cases} \mathcal{O}^{\oplus b_0}_{\mathbb{P}^n}\oplus \Omega^1_{\mathbb{P}^n}(1)^{\oplus k}\oplus \Omega^{n-1}_{\mathbb{P}^n}(n-1)^{\oplus k}\oplus \mathcal{O}_{\mathbb{P}^n}(-1)^{\oplus b_1}& \text{if} \ n\geq 3\\
\mathcal{O}^{\oplus b_0}_{\mathbb{P}^n}\oplus \Omega^1_{\mathbb{P}^n}(1)^{\oplus k}\oplus \mathcal{O}_{\mathbb{P}^n}(-1)^{b_1}& \text{if} \ n=2
\end{cases}\\
\mathcal{M}^{1}&:=\mathcal{O}_{\mathbb{P}^n}^{\oplus b_0-\chi(\mathcal{E})}.
\end{align*}

For a general projective scheme $X$ endowed with an ample and globally generated line bundle $\mathcal{O}_X(h)$, the full linear series of $\mathcal{O}_X(h)$ induces a finite map $\varphi_{|\mathcal{O}_X(h)|}\colon X\to \mathbb{P}^N$ for some $N\in \mathbb{Z}_{>0}$. Then projecting from $N-n$ general points in $\mathbb{P}^N$ induces a finite map $p\colon X\to \mathbb{P}^n$ with the property that $p^*(\mathcal{O}_{\mathbb{P}^n}(1))=\mathcal{O}_X(h)$. Using the fact that $R^ip_*(\mathcal{E})=0$ ($i\geq 1$) for any finite map $p$ and the projection formula, it is proved (\cite[Theorem 1.4]{AC}) that if $\mathcal{E}$ is an $h$-instanton sheaf on $X$, then $p_*(\mathcal{E})$ is a $\mathcal{O}_{\mathbb{P}^n}(1)-$instanton sheaf on $\mathbb{P}^n$. More generally, an $h$-instanton sheaf is preserved by a push-forward along a finite map.

\subsection{Monadic presentations of h-instanton bundles}
We have seen that monadic presentations exist for instanton sheaves on $\pn$ (Theorem \ref{thm:monadic presentation}), perverse instanton sheaves on $\pn$ (Definition \ref{defn:perverse}), and for instanton sheaves on some Fano threefolds (Theorem \ref{thm:monadFano}). 
For a smooth n-fold $X$ ($n\geq 3$), endowed with a very ample line bundle $\calo_X(h)$, assume that $X$ is ACM with respect to this line bundle $\calo_X(h)$. It is proved in \cite[Theorem 1.7]{AC} that a vector bundle $\mathcal{E}$ on $X$ is an $h-$instanton bundle with defect $\delta\in \{0,1\}$ and quantum number $k\in \mathbb{Z}_{\geq 0}$ if and only if it is the cohomology of a monad of certain kind. In this subsection, we recall the monadic presentations for $h-$instanton bundles on such a scheme $X$ with an additional constraint that $h^0(X, \omega_X((n-1)h))=0$. With this extra condition, a monad will be more neatly presented, and it satisfies some duality property. We refer to \cite[Theorem 1.7]{AC}, \cite[Theorem 3.3]{JVM-10} and \cite{CM} for technical details on the construction, and the construction on a general ACM scheme and on a quadric hypersurface.

Firstly, let us recall that a smooth variety $X$ with a very ample line bundle $\mathcal{O}_X(h)$ is ACM if 
\begin{enumerate}
	\item $h^i(\mathcal{O}_X(th))=0$ for $i=1,2,..., n-1$, $t\in \mathbb{Z}$;

	\item $h^i(\mathcal{I}_{X|\mathbb{P}^N}(t))=0$ (where the embedding is $|\mathcal{O}_X(h)|: X\hookrightarrow \mathbb{P}^N$).
\end{enumerate}

A sheaf $\mathcal{E}\in Coh(X)$ is called Ulrich if 
\begin{enumerate}
	\item $h^0(\mathcal{E}(-(t+1)h))=h^n(\mathcal{E}((t-n)h))=0$ for $t\geq 0$ 
	\item $h^i(\mathcal{E}(t))=0$ for $i=1,2,...,n-1$ and $t\in \mathbb{Z}$
\end{enumerate}

If $X$ is an ACM scheme, and it satisfies an additional vanishing condition $h^0(X, \omega_X((n-1))h)=0$, then the monadic presentation of an ordinary h-instanton bundle $\mathcal{E}$ will be in the following form (\cite[Corollary 7.2]{AC})
\begin{equation}{\label{monadulrich}}
0\to \mathcal{C}^{U, h}\to \mathcal{B}\to \mathcal{C} \to 0
\end{equation}
where $\mathcal{C}=\mathcal{O}_X^{\oplus k}$, and $\mathcal{B}$ is a Ulrich bundle. 
The sheaf $\mathcal{C}^{U, h}$ is the Ulrich dual sheaf of $\mathcal{C}$ in the sense that  $\mathcal{C}^{U, h}:= \mathcal{C}^{\vee}((n+1)h+K_X)$.

Indeed, for a smooth scheme $X$, the vanishing condition $h^0(X, \omega_X((n-1))h)=0$ holds only when $X$ falls into the following three cases: $X\cong \mathbb{P}^n$; $X$ is a smooth quadric hypersurface; or $X$ is a scroll over a smooth curve $B$. 

\subsubsection{}
If $X=\mathbb{P}^n$, the monadic presentations for both ordinary (section \ref{section:perverse}) and non-ordinary h-instanton sheaves (section \ref{subsection:h-instOnProj}) are shown in the previous sections.

\subsubsection{}
If $X\subset \mathbb{P}^{n+1}$ is a smooth quadric hypersurface, let $\mathcal{O}_X(h):=\mathcal{O}_{\mathbb{P}^n}(1)|_{X}$, $\mathcal{S}$ (for $n$ odd) and $\mathcal{S}'$, $\mathcal{S}''$ (for $n$ even) be the spinor bundles. Then, depending on the parity of $n$, monad (\ref{monadulrich}) can be written explicitely as one of the following two monads

$$0\to \mathcal{O}_X^{\oplus k} \to S(h)^{\oplus s}\to \mathcal{O}_X(h)^{\oplus k}\to 0 \quad \text{for\ }n= \text{odd} $$
in which $s=h^0(\mathcal{E}\otimes \mathcal{S})-h^1(\mathcal{E}\otimes\mathcal{S})+2^{\left[\frac{n-1}{2}\right]}k$
$$0\to \mathcal{O}_X^{\oplus k} \to S'(h)^{\oplus s'}\oplus S''(h)^{\oplus s''}\to \mathcal{O}_X(h)^{\oplus k}\to 0 \quad \text{for\ } n=\text{even}$$
in which $s'$ and $s''$ are given as follows:
\begin{align*}
s'&=\begin{cases}
h^0(\mathcal{E}\otimes \mathcal{S}')-h^1(\mathcal{E}\otimes\mathcal{S}')+2^{\left[\frac{n-1}{2}\right]}k& \text{if} \ n \equiv 0 \ (\text{mod} 4)\\
h^0(\mathcal{E}\otimes \mathcal{S}'')-h^1(\mathcal{E}\otimes\mathcal{S}'')+2^{\left[\frac{n-1}{2}\right]}k& \text{if} \ n \equiv 2 \ (\text{mod} 4)
\end{cases},\\
s''&=\begin{cases}
h^0(\mathcal{E}\otimes \mathcal{S}'')-h^1(\mathcal{E}\otimes\mathcal{S}'')+2^{\left[\frac{n-1}{2}\right]}k& \text{if} \ n \equiv 0 \ (\text{mod} 4)\\
h^0(\mathcal{E}\otimes \mathcal{S}')-h^1(\mathcal{E}\otimes\mathcal{S}')+2^{\left[\frac{n-1}{2}\right]}k& \text{if} \ n \equiv 2 \ (\text{mod} 4)
\end{cases}.
\end{align*}

For the cases $n=3,4,5$, the above monads coincide with the ones in \cite{F}, \cite{SA} and \cite{O}.

\subsubsection{}
If $X$ is a scroll, i.e., $X=\mathbb{P}(\mathcal{G})$ where $\mathcal{G}$ is a locally free sheaf on a smooth curve $B$, define $\mathcal{O}_X(h):=\mathcal{O}_{\mathbb{P}(\mathcal{G})}(1)$, and let $f$ be the fiber of the projection $\pi\colon X=\mathbb{P}(\mathcal{G})\to B$ at a closed point. In this case, $X$ is a variety of minimal degree in $\mathbb{P}^N$, and Ulrich bundles on such varieties are described in \cite{AHMP}. Furthermore, $X$ is ACM with respect to $\mathcal{O}_X(h)$ only when $B\cong \mathbb{P}^1$. We review two examples for dimensions $n=3$ and $4$ below, and we refer to \cite[Example 7.7 and 7.8]{AC} for more details.

\begin{example}
For $n=3$, monad (\ref{monadulrich}) for an ordinary $h$-instanton bundle will be in the following form
$$
0\to \mathcal{O}_X((d-2)f)^{\oplus k}\to \mathcal{B} \to \mathcal{O}_X(h)^{\oplus k} \to 0
$$
where $\mathcal{B}$ is a Ulrich bundle on $X$. Thanks to \cite[Theorem 4.7]{AHMP}, we know that the Ulrich bundle $\mathcal{B}$ fits into the short exact sequence
$$
0\to \mathcal{B}_2 \to \mathcal{B} \to \mathcal{O}_X((d-1)f)^{\oplus s_3}\to 0,
$$
where $\mathcal{B}_2$ is given as an extension
$$
0\to \mathcal{O}_X(h-f)^{\oplus s_1}\to \mathcal{B}_2 \to \Omega_{X|\mathbb{P}^1}^1(2h-f)^{\oplus s_2} \to 0
$$
for some $s_1, s_2, s_3\in \mathbb{Z}_{>0}$. 
\end{example}

\begin{example} For $n=4$, let $X$ be the image of the Segre embedding $\mathbb{P}^1\times\mathbb{P}^3\hookrightarrow\mathbb{P}^7$. Then $X$ is a rational normal scroll with $\displaystyle\mathcal{G}\cong \bigoplus_{i=0}^{3}\mathcal{O}_{\mathbb{P}^1}(1)$. Let $p\colon X\to \mathbb{P}^3$ be the projection, and we have $\Omega^1_{X|\mathbb{P}^1}\cong p^*(\Omega^1_{\mathbb{P}^3})$.

If $\mathcal{E}$ is an ordinary $h$-instanton bundle with rank $r$ and quantum number $k$, then monad (\ref{monadulrich}) becomes
$$0\to \mathcal{O}_X(2f)^{\oplus k}\to \mathcal{B} \to \mathcal{O}_X(h)^{\oplus k} \to 0$$
where $\mathcal{B}$ fits into the short exact sequence
$$0\to \mathcal{O}_X(h-f)^{\oplus s_1}\oplus p^*\Omega^1_{\mathbb{P}^3}(2h-f)^{\oplus s_2} \to \mathcal{B}\to p^*\Omega^1_{\mathbb{P}^3}(3h-f)^{\oplus s_2}\oplus \mathcal{O}_X(3f)^{\oplus s_4} \to 0$$
for some $s_1, s_2, s_3, s_4 \in \mathbb{Z}_{>0}$. 
\end{example}

\subsection{Examples of $h$-instanton bundles}

In this subsection, we review some constructions of (orientable) $h$-instanton bundles on smooth varieties of low dimensions ($n \leq 3$) and on scrolls. Firstly, recall that a rank 2 $h$-instanton bundle $\mathcal{E}$ on $X$ is orientable if it has defect $\delta\in \{0,1\}$ and $\displaystyle c_1(\mathcal{E})=(n+1-\delta)h+K_X$ in $A^1(X)\cong \Pic(X)$. 
% It is indicated in \cite[Theorem 1.6 and Proposition 6.7]{AC} that a rank 2 $h$-instanton bundle on a smooth variety $X$ is orientable. 

\subsubsection{} When $X$ is a smooth curve endowed with a globally generated ample line bundle $\mathcal{O}_X(h)$, an $h$-instanton sheaf $\mathcal{E}$ is necessarily locally free. Moreover, let $g$ be the genus of $X$, and choose a non-effective divisor $\theta\in \Pic^{g-1}(X)$. Then, by definition  $\mathcal{O}_X(\theta+h)^r$ and $(\mathcal{O}_X(\theta)\oplus \mathcal{O}_X(\theta+h))^r$ are respectively, an ordinary $h$-instanton bundle of rank $r$ and an non-ordinary $h$-instanton bundle of rank $2r$. 

In particular, if $X\cong\mathbb{P}^1$, an $h$-instanton sheaves $\mathcal{E}$ is 
$$\mathcal{E}=\left\{
\begin{array}{ll}
\mathcal{O}_{\mathbb{P}^1}^{\oplus \chi(\mathcal{E})}& \text{if} \ \delta=0 \\
(\mathcal{O}_{\mathbb{P}^1}\oplus \mathcal{O}_{\mathbb{P}^1}(-1))^{\oplus \chi(\mathcal{E})}& \text{if} \ \delta=1 \\
\end{array}
\right.
$$
The quantum number is this case is $k=\delta\chi(\mathcal{E})$.

\subsubsection{} When $X$ is a smooth surface, rank 2 orientable $h$-instanton bundles with defect $\delta\in \{0, 1\}$ and large quantum number always exist. Next, we briefly recall their construction for the case that the Kodaira dimension $\kappa(X)$ is $-\infty$ (\cite[Example 6.10]{AC}), and we refer to \cite[Example 6.10]{AC} and \cite{ESW} for the cases when $\kappa(X)=0,1,2$.

For $\kappa(X)=-\infty$, let $\mathcal{O}_X(h)$ be a very ample line bundle on $X$ with $h^0(X, \mathcal{O}_X(h))=N+1$. A rank $2$ $h$-instanton bundle can be constructed by the Serre correspondence in the following extension
$$
0\to \mathcal{O}_X \to \mathcal{E} \to \mathcal{I}_{Z|X}((1-\delta)h-K_X) \to 0,
$$ 
where $Z\subset X$ is a $0-$dimensional scheme with degree$(Z)\geq (1-\delta)(N+1)+1$.

\subsubsection{} When $X$ is a Fano threefold of Picard rank $1$, define $\mathcal{O}_X(h):=\mathcal{O}_X(H)$ where the divisor $H$ is an ample generator of $\Pic(X)$ (we call $\mathcal{O}_X(H)$ the fundamental line bundle in this case). Denote the index of $X$ by $i_X$ ($i_X=1,2,3,4$). According to \cite[Definition 1.1]{ACG}, we call a rank 2 vector bundle $\mathcal{E}$ on $X$ a \textit{classical instanton bundle} if

\begin{enumerate}
\item $c_1(\mathcal{E})=-\epsilon h$ with $\epsilon\in \{0,1\}$
\item $h^0(\mathcal{E})=h^1(\mathcal{E}(-q^{\epsilon}_Xh))=0$ where $\displaystyle q^{\epsilon}_X:=\left[\frac{i_X+1-\epsilon}{2}\right]$.
\end{enumerate} 

\begin{remark}
Comparing to the other two notions of instanton sheaves on a Fano variety (Definitions \ref{defn:inst-F} and \ref{defn:inst-CJ}), we don't require that $c_1(\mathcal{E})=-e_X$ for $\mathcal{E}$ being a \textit{classical instanton bundle} in this context. There is no constraint on stability of $\mathcal{E}$ either (Remark \ref{rmk:h-inst unstable}).  
\end{remark}

The following proposition shows the relation between \textit{classical instanton bundles} and $h$-instanton bundles of rank $2$.

\begin{proposition} \cite[Proposition 8.6]{AC}
Let $X$ be a Fano threefold of Picard rank 1, endowed with a very ample fundamental line bundle $\mathcal{O}_X(h)$. If $\mathcal{E}$ is a rank 2 vector bundle with $c_1(\mathcal{E})=(4-\delta-i_X)h$ where $\delta =\{0, 1\}$, then the following assertions hold:
\begin{enumerate}
	\item If $\mathcal{E}$ is an h-instanton bundle, then its defect is $\delta$ and 
	\begin{enumerate}
		\item if $(i_X, \delta)\not\in \{(4,0), (4,1), (3,1)\}$, then $\mathcal{E}_{norm, h}$ is a classical instanton bundle, where $\mathcal{E}_{norm, h}:=\mathcal{E}\left(-\left[\frac{c_1(\mathcal{E})+1}{2}\right]h\right)$.
		\item if $(i_X, \delta)\in \{(4,0), (4,1), (3,1)\}$, then $\mathcal{E}_{norm, h}$ is a classical instanton bundle if and only if $h^0(\mathcal{E})=0$.
	\end{enumerate}
	\item If $\mathcal{E}_{norm, h}$ is a classical instanton bundle and 
	\begin{enumerate}
		\item if $(i_X, \delta)\not\in \{(1, 0)\}$, then $\mathcal{E}$ is an h-instanton bundle.
		\item if $(i_X, \delta)\in \{(1, 0)\}$, then $\mathcal{E}$ is an h-instanton bundle with defect $\delta$ if and only if $h^0(\mathcal{E}_{norm, h})=0$.
	\end{enumerate}
\end{enumerate}   
\end{proposition}

\subsubsection{} When $X$ is a scroll of dimension $n\geq 3$ on a smooth curve $B$, rank 2 ordinary $h$-instanton bundles on $X$ with quantum number $k$ can be constructed via the Serre correspondence. Following \cite[Section 10]{AC}, we briefly recall the construction below.

Let $\mathcal{G}$ be a locally free sheaf of rank $r\geq 3$ on $B$. Define $\mathcal{O}_X(h):=\mathcal{O}_{\mathbb{P}(\mathcal{G})}(1)$, and assume that $\mathcal{O}_X(h)$ is ample and globally generated. 
For each $k\in \mathbb{Z}_{\geq 0}$, take $k$ general points $b_i\in B$ ($i=1,2,\dots,k$), and let $L_i=\pi^{-1}(b_i)\cong \mathbb{P}^{n-1}$ be the fibers of the morphism $\pi\colon  X=\mathbb{P}(\mathcal{G})\to B$. Let $\theta\in \Pic^{g-1}(B)$ be a non-effective $\Theta$-characteristic of $B$, and $D$ be a divisor on $B$ such that $\mathcal{O}_B(D)=\det(\mathcal{G})$. 
Then the rank 2 vector bundle $\mathcal{E}$ in the following sequence is proved to be an ordinary orientable $\mu$-semistable $h$-instanton with quantum number $k$:

$$
0\to \mathcal{O}_X(\pi^*(D+\theta))\to \mathcal{E}\to \mathcal{I}_{Z|X}(h+\pi^*(\theta)) \to 0.
$$

\begin{remark}
If $X$ is a Fano threefold with Picard number $\rho_X\geq 2$, then $h$-instanton bundles can behave differently from the classical instanton bundle. We review the case when $X$ is the image of the Segre embedding $\mathbb{P}^1\times \mathbb{P}^1 \times \mathbb{P}^1\hookrightarrow \mathbb{P}^7$. We refer to \cite{CCGM}, \cite{MMPL} and \cite{AM} for classical instanton bundles on Fano varieties of higher Picard rank, and \cite[Section 9]{AC} for other examples of pathologies of $h$-instanton bundles on varieties with higher Picard rank. 

Let $X$ be the image of the Segre embedding $\mathbb{P}^1\times \mathbb{P}^1 \times \mathbb{P}^1\hookrightarrow \mathbb{P}^7$, and $p_i\colon  X\to \mathbb{P}^1$ ($i=1,2,3$) be the three projections. Define $\mathcal{O}_X(h_i):=p_i^*(\mathcal{O}_{\mathbb{P}^1}(1))$, and  $\mathcal{O}_X(h):= \mathcal{O}_X(h_1+h_2+h_3)$. Let $L\subset X$ be the intersection of general sections of $|h_2|$ and $|h_3|$, then $L\cong \mathbb{P}^1$. Choose $s\geq 0$ disjoint such curves, and let $Z$ be their union. There are rank 2 vector bundles $\mathcal{E}$ that fit the following sequence, and one can check that such vector bundles are orientable, ordinary, and simple $h$-instanton bundles but are $\mu$-unstable:
$$
0\to \mathcal{O}_X(h_1+3h_3)\to \mathcal{E}\to \mathcal{I}_{Z|X}(h_1+2h_2-h_3) \to 0.
$$
\end{remark}

%%%%%%%%%%%%%%%%%%%%%%%%%%%%%%%%%%%%%%%%%%%%%%%%%%%%%%%%%%%%%%%%%%%
%%%%%%%%%%%%%%%%%%%%%%%%%%%%%%%%%%%%%%%%%%%%%%%%%%%%%%%%%%%%%%%%%%%

\section{Instanton complexes via Bridgeland stability} \label{section:bridgeland}

For the remainder of the article we continue with the assumptions of Section \ref{section:Fano}, i.e.,  $X$ will denote a smooth projective Fano threefold of Picard number one with $\Pic(X)=\mathbb{Z}H$. As before, we write $K_X=-i_XH$ and $i_X=2q_X+e_X$, where $i_X,q_X,e_X\in\mathbb{Z}$, $i_X>0$, $q_X\geq 0$ and $e_X\in \{0,1\}$. For an object $E\in D^b(X)$ and $\beta\in \mathbb{R}$ we define the numerical twisted Chern character as
$$
v_{\beta}(E)=(\ch_0^{\beta}(E)H^3,\ch_1^{\beta}(E)H^2,\ch_2^{\beta}(E)H,\ch_3^{\beta}(E)),
$$
where
\begin{align*}
    \ch_0^{\beta}(E)&= \ch_0(E) \\
    \ch_1^{\beta}(E)&= \ch_1(E)-\beta \ch_0(E)H\\
    \ch_2^{\beta}(E)&= \ch_2(E)-\beta\ch_1(E)H+\frac{\beta^2}{2}\ch_0(E)H^2 \\
    \ch_3^{\beta}(E)&=\ch_3(E)-\beta\ch_2(E)H+\frac{\beta^2}{2}\ch_1(E)H^2-\frac{\beta^3}{6}\ch_0(E)H^3.
\end{align*}
As shown in \cite{Li}, for every $\beta,\alpha,s\in\mathbb{R}$ with $\alpha,s>0$, the function
$$
Z_{\beta,\alpha,s}(E)=-\ch_3^{\beta}(E)+\left(s+\frac{1}{6}\right)\ch_1^{\beta}(E)H^2+i\left(\ch_2^{\beta}(E)H-\frac{\alpha^2}{2}\ch_0(E)H^3\right)
$$
is the central charge of a stability condition $\sigma_{\beta,\alpha,s}=(Z_{\beta,\alpha,s},\mathcal{A}^{\beta,\alpha})$, whose supporting heart is constructed by the two-step tilting process described below. 
\begin{enumerate}
    \item[(a)] Start by considering the Mumford slope $$
    \mu_{\beta}(E)=\begin{cases}\frac{\ch_1^{\beta}(E)H^2}{\ch_0(E)H^3} & \text{if}\ \ch_0(E)\neq 0\\
    +\infty & \text{otherwise.}\end{cases}
    $$
    The following full additive subcategories of $\Coh(X)$ form a torsion pair:
    \begin{align*}
        \mathcal{F}_{\beta}&=\{E\in\Coh(X)\colon \mu_{\beta}(F)\leq 0\ \text{for all subsheaves}\ F\hookrightarrow E\}\\
        \mathcal{T}_{\beta}&=\{E\in\Coh(X)\colon \mu_{\beta}(Q)> 0\ \text{for all quotient sheaves}\ E\twoheadrightarrow Q\}.
    \end{align*}
     Tilting $\Coh(X)$ with respect to this torsion pair, we obtain the heart of a bounded t-structure on $D^b(X)$:
    $$
    \Coh^{\beta}(X):=\langle \mathcal{F}_{\beta}[1],\mathcal{T}_{\beta}\rangle.
    $$
    \item[(b)] Consider now the tilt slope $$
    \nu_{\beta,\alpha}(E)=\begin{cases}\frac{\ch_2^{\beta}(E)H-\frac{\alpha^2}{2}\ch_0(E)H^3}{\ch_1^{\beta}(E)H^2} & \text{if}\ \ch_1^{\beta}(E)H^2\neq 0\\
    +\infty & \text{otherwise.}\end{cases}
    $$
    As before, we have the torsion pair
    \begin{align*}
        \mathcal{F}_{\beta,\alpha}&=\{E\in\Coh^{\beta}(X)\colon \nu_{\beta,\alpha}(F)\leq 0\ \text{for all subobjects}\ F\hookrightarrow E\ \text{in}\ \Coh^{\beta}(X)\}\\
        \mathcal{T}_{\beta,\alpha}&=\{E\in\Coh^{\beta}(X)\colon \nu_{\beta,\alpha}(Q)> 0\ \text{for all quotients}\ E\twoheadrightarrow Q\ \text{in}\ \Coh^{\beta}(X)\}.
    \end{align*}
     Tilting $\Coh^{\beta}(X)$ with respect to this torsion pair, we obtain the desired heart
    $$
    \mathcal{A}^{\beta,\alpha}:=\langle \mathcal{F}_{\beta,\alpha}[1],\mathcal{T}_{\beta,\alpha}\rangle.
    $$
\end{enumerate}
We denote the corresponding Bridgeland slope on $\mathcal{A}^{\beta,\alpha}$ by
$$
\lambda_{\beta,\alpha,s}(E)=\frac{\ch_3^{\beta}(E)-\left(s+\frac{1}{6}\right)\alpha^2\ch_1^{\beta}(E)H^2}{\ch_2^{\beta}(E)H-\frac{\alpha^2}{2}\ch_0(E)H^3}.
$$
%Let us recall Faenzi's definition of instanton sheaf:

%\begin{definition}[\cite{F}]\label{FaenziInstanton} $E\in\Coh(X)$ is called an instanton sheaf if it satisfies the following:
%\begin{itemize}
 %   \item[(i)] E is Gieseker semistable,
  %  \item[(ii)] $E\in \langle \mathcal{O}(q_X)\rangle^{\perp}$, 
  %  \item[(iii)] $\ch(E)$ is fixed by the functor $E\mapsto R\mathcal{H}om(E,\mathcal{O}(-e_X))$.
%\end{itemize}
%\end{definition}
We want to propose a new definition of instanton object that ``extends'' the definitions by Faenzi \cite{F}, Kuznetsov \cite{K}, and Comaschi--Jardim \cite{CJ} included in Section \ref{section:Fano} and that can be satisfied by some objects in $D^b(X)$. We hope that our categorical approach will allow for a more systematical way of studying instanton moduli spaces. The idea will be to replace the $\mu$-semistablity and the vanishing conditions by some type of Bridgeland stability. To determine the appropriate stability condition recall that in the classical case of rank 2 instanton bundles on $\mathbb{P}^3$, the vanishing conditions giving the monad description are obtained after combining $\mu$-stability with the vanishing of one cohomology group, Serre duality, and the fact that a rank 2 vector bundle with trivial first Chern class is self-dual. 

Let us start by consider the functor 
\begin{equation}\label{dualityFunctor}
E\mapsto E^D:=R\mathcal{H}om(E,\mathcal{O}(-e_X))[2].
\end{equation}
We have the following:
\begin{proposition}\label{duality} \cite[Proposition 6.12]{JMM}. 
Suppose that $\nu_{\beta,\alpha}(E)\neq 0$, then
$$
E\ \text{is}\ \sigma_{\beta,\alpha,s}-\text{semistable}\ \Longleftrightarrow\ E^D\ \text{is}\ \sigma_{-\beta-e_X,\alpha,s}-\text{semistable}. 
$$
In particular, if $\beta_0:=-e_X/2$ and $\nu_{\beta_0,\alpha}(E)\neq 0$, then $E$ is $\sigma_{\beta_0,\alpha.s}$-semistable if and only if $E^D$ is $\sigma_{\beta_0,\alpha,s}$-semistable.
\end{proposition}
\begin{proposition}
If $\ch(E)=\ch(E^D)$ then $v_{\beta_0}(E)=(-R,0,D,0)$. Additionally, if $E\in \mathcal{A}^{\beta_0,\alpha}$ for all $\alpha>0$ then $R,D\geq 0$.
\end{proposition}
\begin{proof}
If $v=(v_0,v_1,v_2,v_3)$, define $v^{\vee}:=(v_0,-v_1,v_2,-v_3)$. Then it is clear that
$$
v_{\beta}(E)^{\vee}=v_{-\beta}(E^{\vee}).
$$
Thus
$$
v_{\beta_0}(E)^{\vee}=v_{-\beta_0-e_X}(E^{\vee}\otimes \mathcal{O}(-e_X))=v_{\beta_0}(E^D)=v_{\beta_0}(E).
$$
Additionally, if $E\in\mathcal{A}^{\beta_0,
\alpha}$ for all $\alpha>0$ then 
$$
D+\frac{\alpha^2}{2}R\geq 0\ \ \text{for all}\ \alpha>0,
$$
implying that $R,D\geq 0$.
\end{proof}

\begin{proposition}\label{stab:linebundles}
The line bundle $\calo(q_X)$ is $\sigma_{\beta_0,\alpha,s}$-stable for $\alpha<i_X/2$ and $s>0$. 
\end{proposition}
\begin{proof}
This is a straightforward computation as line bundles on a threefold of Picard rank 1 are tilt stable and Bridgeland stable as long as they belong to $\Coh^{\beta_0}(X)$ and $\mathcal{A}^{\beta_0,\alpha}$,  respectively. Notice that
$$
v_{\beta_0}(\calo(q_X))=\left(H^3,\frac{i_X}{2}H^3,\frac{1}{2}\left(\frac{i_X}{2}\right)^2H^3,\frac{1}{6}\left(\frac{i_X}{2}\right)^3H^3\right),
$$
and so 
$$
\mu_{\beta_0}(\calo(q_X))=\frac{i_X}{2},\ \ \nu_{\beta_0,\alpha}(\calo(q_X))=\frac{\frac{1}{2}\left(\frac{i_X}{2}\right)^2-\frac{\alpha^2}{2}}{\frac{i_X}{2}}.
$$
Thus $\calo(q_X)\in \Coh^{\beta_0}(X)$ because $\calo(q_X)$ is $\mu_{\beta_0}$-stable. Therefore, $\calo(q_X)$ is also $\nu_{\beta_0,\alpha}$-stable. Since $\nu_{\beta_0,\alpha}(\calo(q_X))>0$ if and only if $\alpha<i_X/2$, then for these values $\calo(q_X)\in \mathcal{A}^{\beta_0,\alpha}$.
\end{proof}
Consider the region
$$
\mathcal{U}:=\left\{(\alpha,s)\colon \alpha,s>0,\ \left(s+\frac{1}{6}\right)\alpha^2<\frac{1}{6}\left(\frac{i_X}{2}\right)^2\right\}.
$$
We see $\mathcal{U}\subset \mathrm{Stab}(X)$ via the identification $(\alpha,s)\leftrightarrow \sigma_{\beta_0,\alpha,s}$. We refer to the set of stability conditions $\{\sigma_{\beta_0,\alpha,s}\}_{\alpha,s>0}$ as the $(\alpha,s)$-slice. From the computations above, it is clear that $\lambda_{\beta_0,\alpha,s}(\calo(q_X))>0$ for all $(\alpha,s)\in \mathcal{U}$.

\begin{definition}\label{instantonComplex}
Fix a numerical twisted Chern character $v_{\beta_0}=(-R,0,D,0)$ with $R\geq 0$, $D>0$, and let $\mathcal{C}$ be a chamber for $v_{\beta_0}$ such that $\mathcal{C}\cap\mathcal{U}\neq \emptyset$. Let $E\in D^b(X)$ with $v_{\beta_0}(E)=v_{\beta_0}$, we say that $E$ is a  $\mathcal{C}$-instanton object if $E$ is Bridgeland semistable for a stability condition in $\mathcal{C}$.
\end{definition}
\begin{remark}
It was proven in \cite[Section 5]{JMM} that for the numerical twisted Chern character $v_{\beta_0}=(-R,0,D,0)$, the wall and chamber decomposition of the $(\alpha,s)$-slice is finite. Moreover, an algorithm to compute the walls was provided.
\end{remark}

\begin{lemma}\label{limit:inst}
Let $E$ be a $\calc$-instanton that is also Bridgeland semistable in the outermost chamber of the $(\alpha,s)$-slice. Then if $(\alpha_0,s_0)\in\calc$, $E$ is $\sigma_{\beta_0,\alpha,s}$-semistable for all $\alpha\geq \alpha_0$, i.e., in all the chambers to the right of $\calc$.
\end{lemma}
\begin{proof}
Since the walls are disjoint and each destabilizing wall for $E$ intersects the vertical line $\{(\alpha_0,s)\colon s>0\}$ then it is enough to prove that no subobject of $E$ in the category $\mathcal{A}^{\beta_0,\alpha_0}$ can destabilize $E$ for $s>s_0$. Indeed, if $A\hookrightarrow E$ is such subobject then
\begin{align*}
    \lambda_{\beta_0,\alpha_0,s}(A)&\leq \lambda_{\beta_0,\alpha_0,s}(E)\ \ \text{in the outermost chamber, and}\\
    \lambda_{\beta_0,\alpha_0,s}(A)&=\lambda_{\beta_0,\alpha_0,s}(E)=0\ \ \text{at the wall produced by}\ A.
\end{align*}
Thus, $\lambda_{\beta_0,\alpha_0,s_0}(A)>\lambda_{\beta_0,\alpha_0,s_0}(E)$ because the numerator of $\lambda_{\beta_0,\alpha_0,s}$ is linear in $s$. A contradiction, unless $A$ never really destabilizes $E$, i.e., $\lambda_{\beta_0,\alpha_0,s}(A)=0$ for all $s\geq s_0$.
\end{proof}
\begin{theorem}\label{mainVanishing}
Let $E\in D^b(X)$ be an object with $v_{\beta_0}(E)=(-R,0,D,0)$ and $\mathcal{C}$ be a chamber for this Chern character in the $(\alpha,s)$-slice such that $\mathcal{C}\cap\mathcal{U}\neq \emptyset$. Then
\begin{enumerate}
    \item $E$ is a $\mathcal{C}$-instanton object if and only if $E^D$ is a $\mathcal{C}$-instanton object.
    \item If $E$ is a $\mathcal{C}$-instanton object then $E\in \langle\calo(q_X)\rangle^{\perp}$.
\end{enumerate}
\end{theorem}
\begin{proof}
Part (1) is a direct consequence of Proposition \ref{duality} and the fact that $v_{\beta_0}(E)=v_{\beta_0}(E^D)$. For part (2) notice that
$$
\Hom(\calo(q_X),E)=0
$$
since for stability conditions on $\mathcal{C}$, $\calo(q_X)$ is stable with 
$$
\lambda_{\beta_0,\alpha,s}(\calo(q_X))>0=\lambda_{\beta_0,\alpha,s}(E).
$$
Notice that
\begin{align*}
    \Ext^i(\calo(q_X),E)&=\Ext^{3-i}(E,\calo(q_X-i_X))^*\\
    &=\Ext^{3-i}(\calo,E^{\vee}\otimes \calo(-e_X)\otimes \calo(q_X-i_X+e_X))^*\\
    &=\Ext^{3-i}(\calo,E^D\otimes\calo(-q_X)[-2])^*\\
    &=\Ext^{1-i}(\calo(q_X),E^D)^*.
\end{align*}
Combining this computation with part (1) we get $\Ext^1(\calo(q_X),E)=0$. For other values of $i$ we get a negative Ext between objects in the same heart, which is impossible. Therefore $E\in \langle\calo(q_X)\rangle^{\perp}$.
\end{proof}
%%%%%%%%%%%%%%%%%%%%%%%%%%%%%%%%%%%%%%%%%%%%%%%%%%%%%%%%%%%%%%%%%%%
%%%%%%%%%%%%%%%%%%%%%%%%%%%%%%%%%%%%%%%%%%%%%%%%%%%%%%%%%%%%%%%%%%%
\begin{example}
It was shown in \cite[Proposition 5.3]{JMM} that the only Bridgeland semistable objects of twisted Chern character $v_{\beta_0}=(0,0,D,0)$ with $D>0$ in the outermost chamber of the corresponding finite wall and chamber decomposition of the $(\alpha,s)$-slice are precisely the (twisted) Gieseker semistable sheaves. Moreover, if a (twisted) Gieseker semistable sheaf $T$ with $v_{\beta}(T)=(0,0,D,0)$ is also a  $\mathcal{C}$-instanton object, i.e., does not get destabilized at any point before the potential numerical wall produced by $\calo(q_X)$, then Theorem \ref{mainVanishing} shows that $E$ is a 1-dimensional instanton sheaf.  
\end{example}
\begin{example}\label{spinorEx}
Let $X=Q^3\subset \mathbb{P}^4$ be a quadric hypersurface and $\iota\colon X\hookrightarrow \mathbb{P}^4$ the corresponding inclusion. The spinor bundle $S$ on $X$ is defined by the short exact sequence
$$
0\lra \calo_{\p3}(-1)^{\oplus 4}\stackrel{M}{\lra} \calo_{\p3}^{\oplus 4}\stackrel{N}{\lra} \iota_*S\lra 0,
$$
where the matrix $M$ satisfies
$$
M^2=(x_0^2+x_1x_2+x_3x_4)I.
$$
Moreover, restricting $N$ to $X$ produces the short exact sequence
\begin{equation}\label{spinor}
0\lra S(-1)\lra \calo_X^{\oplus 4}\lra S\lra 0.
\end{equation}
The sheaf $S(-1)$ is an instanton sheaf of rank 2 and minimal charge. Notice that in this case $i_X=3$ and so $q_X=1=e_X$. Thus $\beta_0=-1/2$ and a simple computation using the exact sequence \eqref{spinor} leads to
$$
v_{\beta_0}(S(-1)[1])=\left(-2H^3,0,\frac{H^3}{4},0\right)=\left(-4,0,\frac{1}{2},0\right).
$$
Notice that since $S(-1)$ is slope stable with  $\mu_{\beta_0}(S(-1))=0$ then $S(-1)[1]\in \Coh^{\beta_0}(X)$, and since $\nu_{\beta_0,\alpha}(S(-1)[1])=+\infty$ then $S(-1)[1]\in \mathcal{A}^{\beta_0,\alpha}$ for all $\alpha>0$.

Now, it follows from \cite[Theorem 3.1]{JMM} that $S(-1)[1]$ is asymptotically $\lambda_{\beta_0,\alpha,s}$-stable and so Bridgeland stable in the outermost chamber of the $(\alpha,s)$-slice. Moreover, as proven in \cite[Section 5.3]{JMM}, if $E$ is an object with $v_{\beta_0}(E)=(-R,0,D,0)$ and $A\hookrightarrow E$ is a destabilizing subobject producing a 1-dimensional wall in the $(\alpha,s)$-slice, then if $v_{\beta_0}(A)=(r,c,d,e)$ we must have 
\begin{align}
   & 0<d<D,\label{ineq:d}\\
    &0<c(6e)\leq \min\{(2d)^2,(2D-2d)^2\},\label{ineq:c6e}\\
    &-\frac{c}{6e}(2D-2d)-R\leq r\leq \frac{c}{6e}2d.\label{ineq:r} 
\end{align}
Besides, in our case we also have
\begin{equation}\label{integralcond}
\ch_1(A)H^2\in 2\mathbb{Z},\ \ch_2(A)H\in \mathbb{Z},\ r\in 2\mathbb{Z},\ 4d\in\mathbb{Z},\ \text{and}\ 24e\in\mathbb{Z}. 
\end{equation}
Thus, using $v_{\beta_0}(S(-1)[1])=\left(-4,0,\frac{1}{2},0\right)$, a straightforward computation shows that the only possibilities for the Chern character of a destabilizing subobject of $S(-1)[1]$ are
$$
v_{\beta_0}(A)=\left(r,1,\frac{1}{4},\frac{1}{24}\right),\ \ r=-6,\ 2,
$$
which produce the only potential destabilizing wall:
$$
W=\left\{(\alpha,s)\colon \left(s+\frac{1}{6}\right)\alpha^2=\frac{1}{24}\right\}.
$$
Thus, if $\mathcal{C}_{out}$ denotes the outermost chamber for $v_{\beta_0}(S(-1)[1])$ in the $(\alpha,s)$-slice, then $\mathcal{C}_{out}\cap \mathcal{U}\neq \emptyset$ and so $S(-1)[1]$ is a $\mathcal{C}_{out}$-instanton object.
\end{example}

\begin{example}\label{ex:idealsheaf}
Let $X$ be a Fano threefold of Picard number one, index 2 and degree $H^3>1$. In this case, $e_X=0$ and $q_X=1$. If $\ell\subset X$ is a line then $\ch(\mathcal{I}_{\ell}[1])=(-1,0,\ell,0)$ and so
$$
v_0(\mathcal{I}_{\ell}[1])=(-H^3,0,1,0).
$$
It follows from \cite[Example 3.4]{JMM} and the Gieseker stability of $\mathcal{O}_{\ell}$ that $\mathcal{I}_{\ell}[1]$ is $\sigma_{0,\alpha,s}$-stable for all $\alpha\gg 0$. On the other hand, if $\mathcal{I}_{\ell}[1]$ is ever unstable in the $(\alpha,s)$-slice then there should be a destabilising sub-object $A\hookrightarrow \mathcal{I}_{\ell}[1]$ with $v_0(A)=(r,c,d,e)$ satisfying inequalities \eqref{ineq:d}, \eqref{ineq:c6e}, and \eqref{ineq:r}. Thus, we should have 
$$
0<2d<2
$$
and so $2d=1$, $c=1$, and $6e=1$. However, $c=\ch_1(A)H^2\neq 1$ because $H^3>1$ and so such $A$ can not exist. Therefore, $\mathcal{I}_{\ell}[1]$ is a $\mathcal{C}$-instanton object for each chamber $\mathcal{C}$ such that $\mathcal{C}\cap \mathcal{U}\neq \emptyset$.
\end{example}

\begin{example} 
It was noticed in \cite[Example 6.15]{JMM} that the  object $A\in D^b(\mathbb{P}^3)$ defined by the exact triangle
\begin{equation}\label{nonsheaf:inst}
\calo_H(-1)[1] \longrightarrow A \longrightarrow \calo_H(2), 
\end{equation}
where $H$ denotes a hyperplane in $\p3$, is Bridgeland stable in the stability chamber that contains the rank 0 instanton sheaves as stable objects. Thus, such $A$ is a $\mathcal{C}$-instanton object in our definition. We would like to directly check that $A$ has indeed a ``monadic'' presentation. 

Our starting point is the Euler sequence for the cotangent bundle of $H$:
$$ 0 \lra \Omega^1_H(2) \lra \calo_H(1)^{\oplus 3} \lra \calo_H(2) \lra 0 $$
Composing the epimorphism above with $\calo_{\p3}(1)^{\oplus 3}\onto\calo_H(1)^{\oplus 3}$ we obtain an exact sequence
\begin{equation}\label{goodG}
    0 \lra G \lra \calo_{\p3}(1)^{\oplus 3} \lra \calo_H(2) \lra 0,
\end{equation}
with the sheaf $G$ being given by the following extension
$$ 0 \lra \calo_{\p3}^{\oplus 3} \lra G \lra \Omega_H^1(2) \lra 0. $$
Now there is a composed epimorphism $\op3^{\oplus 3}\onto\calo_H^{\oplus 3}\onto\Omega_H^1(2)$
whose kernel $K$ is given by the exact sequence
$$ 0 \lra \op3(-1)^{\oplus 3} \lra K \lra \calo_H(-1) \lra 0, $$
and also satisfies the sequence
\begin{equation}%\label{coevaluation}
0\lra K \lra \op3^{\oplus 6} \lra G\lra 0.
\end{equation}
We then obtain a monomorphism
$$ \alpha\colon \op3(-1)^{\oplus 3} \into K \into \op3^{\oplus 6}  $$
and a morphism
$$ \beta\colon \op3^{\oplus 6} \onto G \into \op3(1)^{\oplus 3} $$
whose cokernel is precisely $\calo_H(2)$. Moreover, $\beta\alpha=0$, and $(\ker\beta/\im\alpha)\simeq\calo_H(-1)$, since $\ker\beta=K$ and $\im\alpha\simeq\op3(-1)^{\oplus 3}$. 

In summary, the object $A$ is quasi-isomorphic to the complex
$$ \op3(-1)^{\oplus 3} \stackrel{\alpha}{\lra} \op3^{\oplus 6} \stackrel{\beta}{\lra} \op3(1)^{\oplus 3}. $$
\end{example}
\begin{remark}
The $\mathcal{C}$-instanton object $A$ defined by the exact triangle \eqref{nonsheaf:inst} is a new type of instanton. Indeed, the only instanton complexes previously defined were the perverse instantons and $A$ does not fall into this category since $\mathcal{H}^{-1}(A)=\calo_H(-1)$ is not a torsion free sheaf (see Definition \ref{defn:perverse}).
\end{remark}
\subsection{Monad descriptions from quiver regions} In this subsection we will study two instances, namely $\mathbb{P}^3$ and the quadric $Q^3\subset \mathbb{P}^4$, in which we have full strong exceptional collections producing quiver regions intercepting the corresponding $(\alpha,s)$-slice. This will then lead to monad-type descriptions for some $\mathcal{C}$-instanton objects. We remark that similar techniques may be used to study other Fano threefolds with full strong exceptional collections.

\subsubsection{$X=\mathbb{P}^3$} In this case we have $e_X=0$ and so $\beta_0=0$. A simple computation as in \cite[Lemma 6.10]{JMM} shows that for 
$$
(\alpha,s)\in \mathcal{R}=\left\{(\alpha,s)\colon 1<(6s+1)\alpha^2<\frac{4-3\alpha^2}{2-\alpha^2},\ 0<\alpha<1,\ s>0\right\}
$$
we have $\calo_{\p3}(-2)[2],\calo_{\p3}(-1)[2],\calo_{\p3}[1],\calo_{\p3}(1)\in\mathcal{A}^{0,\alpha}$ and moreover
$$
\lambda_{0,\alpha,s}(\calo_{\p3}(-2)[2])<\lambda_{0,\alpha,s}(\calo_{\p3}(1))\leq 0=\lambda_{0,\alpha,s}(\calo_{\p3}[1])\leq \lambda_{0,\alpha,s}(\calo_{\p3}(-1)[2]).
$$
Thus, for each $(\alpha,s)\in \mathcal{R}$ we can choose $t_{\alpha,s}\in\mathbb{R}$ such that
$$
\lambda_{0,\alpha,s}(\calo_{\p3}(-2)[2])<t_{\alpha,s}<\lambda_{0,\alpha,s}(\calo_{\p3}(1)),
$$
so that tilting $\mathcal{A}^{0,\alpha}$ with respect to the torsion pair
\begin{align*}
        \mathcal{F}_{0,\alpha,s}&=\{E\in\mathcal{A}^{0,\alpha}\colon \lambda_{0,\alpha,s}(F)\leq t_{\alpha,s}\ \text{for all subobjects}\ F\hookrightarrow E\ \text{in}\ \mathcal{A}^{0,\alpha}\}\\
        \mathcal{T}_{0,\alpha,s}&=\{E\in \mathcal{A}^{0,\alpha}\colon \lambda_{0,\alpha,s}(Q)> t_{\alpha,s}\ \text{for all quotients}\ E\twoheadrightarrow Q\ \text{in}\ \mathcal{A}^{0,\alpha}\},
    \end{align*}
we obtain the heart $\langle \calo_{\p3}(-2)[3],\calo_{\p3}(-1)[2],\calo_{\p3}[1],\calo_{\p3}(1)\rangle$.

Now, notice that every Bridgeland semistable object $E$ with $\ch(E)=(-R,0,D,0)$ is in the subcategory $\mathcal{T}_{0,\alpha,s}$ and so it is quasi-isomorphic to a complex of the form 
\begin{equation}\label{quiver:p3}
0\lra \calo_{\p3}(-1)^{\oplus D}\lra \calo_{\p3} ^{\oplus 2D+R}\lra \calo_{\p3}(1)^{\oplus D}.
\end{equation}
Moreover, it was also established in \cite[Lemma 6.6]{JMM} that the Bridgeland stability of $E$ is equivalent to King stability of the complex \eqref{quiver:p3} with respect to the vector $\Theta_0=(-1,0,1)$. In particular, none such objects can be Bridgeland unstable in the region $\mathcal{R}$. Therefore there is only one Bridgeland chamber $\mathcal{C}$ for $v_0=(-R,0,D,0)$ intersecting the region $\mathcal{R}$, and all the $\mathcal{C}$-instanton objects in $\p3$ have the monad-type description \eqref{quiver:p3}.
\begin{center}
    \begin{figure}[H]
        \centering
        \includegraphics[scale=0.45]{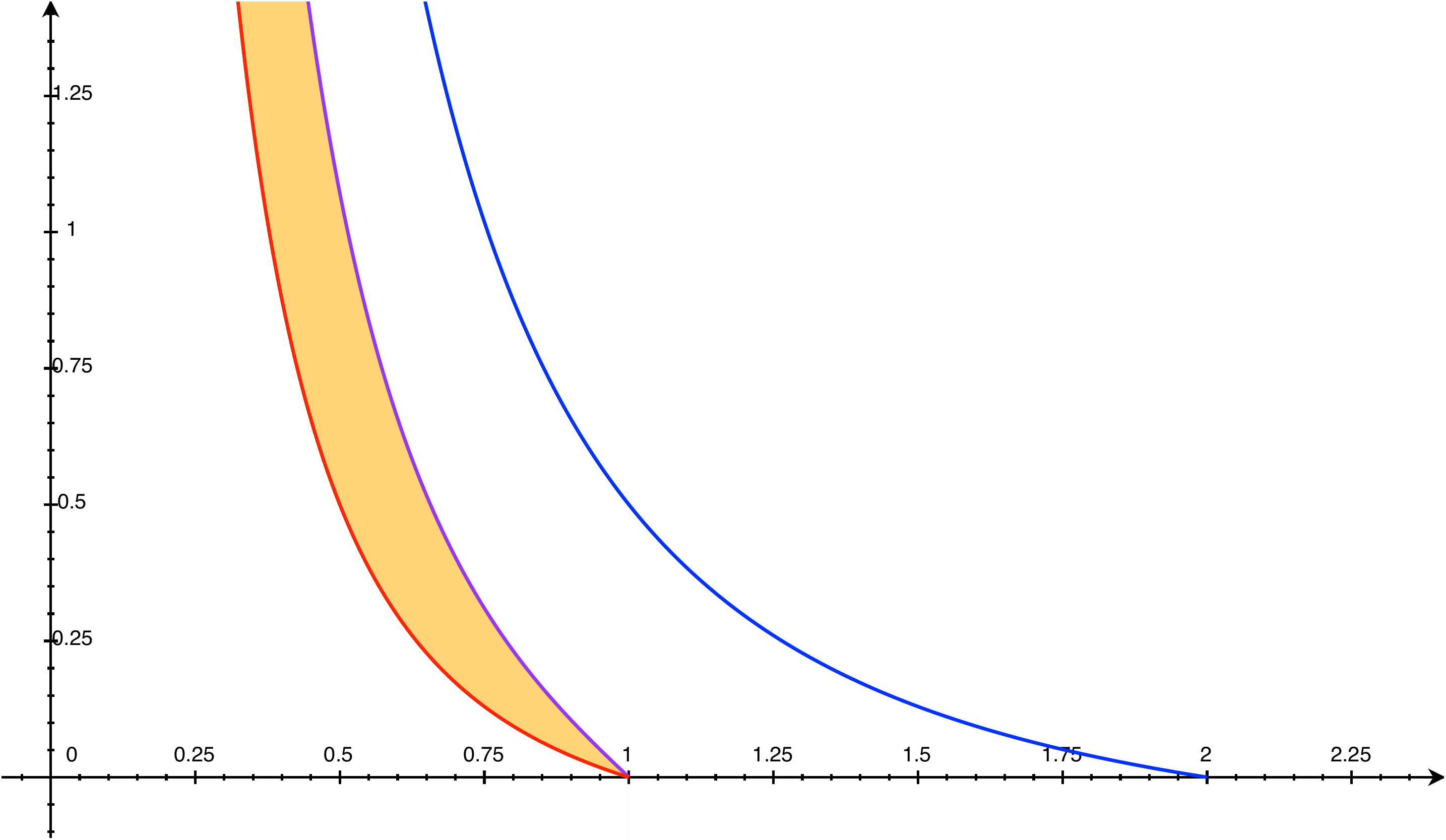}
        \caption{The quiver region $\mathcal{R}$ is pictured in the $(\alpha,s)$-slice in yellow. The blue curve is the potential wall for $\calo_{\p3}(2)$ (the instanton wall), the red curve is the potential wall for $\calo_{\p3}(1)$ (the collapsing wall), and the purple curve is given by the equation of Bridgeland slopes  $\lambda_{0,\alpha,s}(\calo_{\p3}(2)[2])=\lambda_{0,\alpha,s}(\calo_{\p3}(1))$.}
        \label{fig:quiverregion}
    \end{figure}
\end{center}
\begin{remark}\label{vector:quiverR}
As explained in Section \ref{sec:repQuiver}, in \cite{JS} the authors studied in detail the variation of GIT for representations of the quiver
\begin{equation}\label{quiverP3}
\begin{tikzcd}
\bullet \arrow[rr,bend left,"\alpha_{0}"] \arrow[rr,bend right,swap,"\alpha_{3}"] &\vdots & \bullet \arrow[rr,bend left,"\beta_{0}"] \arrow[rr,bend right,swap,"\beta_{3}"] &\vdots &\bullet 
\end{tikzcd}
\end{equation}
for the dimension vector $\vec{n}=(1,3,1)$, i.e., $R=2$ and $D=1$. They proved that King's space of stability parameters $\vec{n}^{\perp}$ has a finite wall and chamber decomposition where only two chambers have associated nonempty moduli spaces. These chambers share a wall given by the vector $\Theta_0:=(-1,0,1)\in \vec{n}^{\perp}$. Moreover, they proved that the $\Theta_0$-semistable representations are S-equivalent to complexes \eqref{quiver:p3} quasi-isomorphic to shifts of instanton sheaves (both locally and non-locally free) or to perverse instantons.
\end{remark}
\begin{example}
Consider the unique non split extension
$$
0\lra \calo_{\p3}(-2)[2]\lra E\lra \calo_{\p3}(2)\lra 0.
$$
At the ``instanton wall'' given by   $\lambda_{0,\alpha,s}(\calo_{\p3}(2))=0$, the object $E$ is strictly semistable and so it is stable in a chamber $\calc_0$ intersecting the region $\mathcal{U}$. Notice that even though the object $E$ is self-dual, i.e., $E^D=E$, this is not a $\calc_0$-instanton in our definition since its Chern character  $\ch(E)=(2,0,2,0)$ does not satisfy the usual Bogomolov inequality. Moreover, the chamber $\calc_0$ can not intersect the quiver region $\mathcal{R}$ since such $E$ does not have a monad-type description. Indeed, if that were the case then $E$ would be quasi-isomorphic to a complex of the form
$$
\calo_{\p3}(-1)^{\oplus 2}\lra \calo_{\p3}^{\oplus 2}\lra \calo_{\p3}(1)^{\oplus 2},
$$
but this would imply that there is a surjective map $\calo_{\p3}(1)^{\oplus 2}\twoheadrightarrow \calo_{\p3}(2)$, which is impossible.
\end{example}
\begin{example}
Let $E$ be a rank 2 instanton bundle in $\p3$ as defined in Section 2, in particular $E$ is $\mu$-stable. Then the classical monad presentation of $E$ tells us that $E[1]$ is quasi-isomorphic to a complex of the form
$$
\calo_{\p3}^{\oplus c}\lra \calo_{\p3}^{\oplus 2+2c}\lra \calo_{\p3}(1)^{\oplus c}. 
$$
Moreover, by \cite[Proposition 6]{JS} this representation of the quiver $\mathbf{Q}$ in display \eqref{Q} is stable with respect to the vector $\Theta_0=(-1,0,1)$. On the other hand, due to the $\mu$-stability and the locally freeness of $E$, the object $E[1]$ is stable in the outermost chamber as it satisfies the hypotheses of \cite[Theorem 3.1]{JMM}. Therefore, Lemma \ref{limit:inst} implies that $E[1]$ is a $\calc$-instanton object for every chamber $\calc$ such that $\calc\cap\mathcal{U}\neq \emptyset$. Clearly, the same argument shows that if $E$ is $\mu$-stable locally-free instanton sheaf on $\p3$ then $E[1]$ is a $\calc$-instanton object.
\end{example}
\subsubsection{$X=Q^3\subset\p4$} In this case we have $e_X=1$ and so $\beta_0=-1/2$. We have the full strong exceptional collection
$$
\calo_X(-1),\ S(-1),\ \calo_X,\ \calo_X(1),
$$
where $S(-1)$ is the spinor bundle defined in Example \ref{spinorEx}. Notice that since $H^3=2$ then
$$
v_{\beta_0}(\calo_X(n))=\left(2,2\left(n+\frac{1}{2}\right),\left(n+\frac{1}{2}\right)^2,\frac{1}{3}\left(n+\frac{1}{2}\right)^3\right).
$$
Thus, computing the signs of the Mumford and tilt slopes we get  
\begin{align*}
\calo_X(-1)[1],\ \calo_X,\ \calo_X(1)\in \Coh^{\beta_0}(X),&\\
    \calo_X(-1)[2]\in \mathcal{A}^{\beta_0,\alpha}\ \ \text{for}\ \alpha\leq \frac{1}{2}&,\\
    \calo_X\in \mathcal{A}^{\beta_0,\alpha}\ \ \text{for}\ \alpha< \frac{1}{2}&,\\
    \calo_X(1)\in \mathcal{A}^{\beta_0,\alpha}\ \ \text{for}\ \alpha< \frac{3}{2}&.
\end{align*}
For $(\alpha,s)$ in the region
$$
\mathcal{R}=\left\{(\alpha,s)\colon \frac{1}{4}<(6s+1)\alpha^2<\frac{9}{4},\ \ 0<\alpha<\frac{1}{2},\ \ s>0\right\}
$$
we have 
$$
\lambda_{\beta_0,\alpha,s}(\calo_X)<\lambda_{\beta_0,\alpha,s}(S(-1)[1])=0<\lambda_{\beta_0,\alpha,s}(\calo_X(-1)[2]),\ \ \text{and}\ \ \lambda_{\beta_0,\alpha,s}(\calo_X(1))>0,
$$
and from Example \ref{spinorEx}, it follows that all these objects are Bridgeland semistable in $\mathcal{R}$. Additionally, since within $\mathcal{R}$,  $\lambda_{\beta_0,\alpha,s}(\calo_X(1))$ approaches zero as we get closer to $(6s+1)\alpha^2=\frac{9}{4}$ while $\lambda_{\beta_0,\alpha,s}(\calo_X(-1)[2])$ approaches zero as we get closer to $(6s+1)\alpha^2=\frac{1}{2}$, then the region
$$
\mathcal{R}_0=\left\{(\alpha,s)\colon \lambda_{\beta_0,\alpha,s}(\calo_X(-1)[2])<\lambda_{\beta_0,\alpha,s}(\calo_X(1)),\ \alpha,s>0\right\}
$$
has a nonempty intersection with $\mathcal{R}$. Therefore, for each $(\alpha,s)\in \mathcal{R}\cap\mathcal{R}_0$ we can choose $t_{\alpha,s}\in\mathbb{R}$ such that
$$
\lambda_{\beta_0,\alpha,s}(\calo_X(-1)[2])<t_{\alpha,s}<\lambda_{\beta_0,\alpha,s}(\calo_X(1)),
$$
and tilting the heart $\mathcal{A}^{\beta_0,\alpha}$ with respect to the torsion pair
\begin{align*}
        \mathcal{F}_{\beta_0,\alpha,s}&=\{E\in\mathcal{A}^{\beta_0,\alpha}\colon \lambda_{\beta_0,\alpha,s}(F)\leq t_{\alpha,s}\ \text{for all subobjects}\ F\hookrightarrow E\ \text{in}\ \mathcal{A}^{\beta_0,\alpha}\}\\
        \mathcal{T}_{\beta_0,\alpha,s}&=\{E\in \mathcal{A}^{\beta_0,\alpha}\colon \lambda_{\beta_0,\alpha,s}(Q)> t_{\alpha,s}\ \text{for all quotients}\ E\twoheadrightarrow Q\ \text{in}\ \mathcal{A}^{\beta_0,\alpha}\},
    \end{align*}
we obtain the heart $\langle \calo_X(-1)[3],S(-1)[2],\calo_X[1],\calo_X(1)\rangle$.

Therefore, it follows that if $E$ is a  $\lambda_{\beta_0,\alpha,s}$-semistable object for $(\alpha,s)\in \mathcal{R}\cap\mathcal{R}_0$ with  $v_{\beta_0}(E)=(-R,0,D,0)$, then $E[1]$ is quasi-isomorphic to a complex of the form
$$
\calo_X(-1)^{\oplus a}\lra S(-1)^{\oplus b} \lra \calo_X^{\oplus c}\lra \calo_X(1)^{\oplus n}.
$$
A straightforward computation of the Chern characters of this complex gives us 
$$
n=0,\ a=c=D-\frac{R}{8}=\ch_2(E)H,\ \text{and}\  b=c+\frac{R}{4}.
$$
Thus, if $\mathcal{C}$ is a chamber for $v_{\beta_0}=(-R,0,D,0)$ intersecting the region $\mathcal{R}\cap\mathcal{R}_0$, then every $\mathcal{C}$-instanton object has a monad-type description of the form
$$
\calo_X(-1)^{\oplus c}\lra S(-1)^{\oplus c+\frac{R}{4}} \lra \calo_X^{\oplus c}.
$$

\subsection{Acyclic extensions revisited} Let $X$ be a Fano threefold of Picard number one and index 2. In this case $q_X=1$ and $e_x=0$. As in the case of rank 2 instanton bundles we have the following.
\begin{lemma}\label{vanishing:Cinst}
Let $E$ be an stable $\mathcal{C}$-instanton object with $v_0(E)=(-R,0,D,0)$. Then
$$
\dim\Ext^{i}(\calo_X[1],E)=\begin{cases} D-\frac{R}{H^3} &\text{if}\ i=1\\
0&\text{otherwise}.
\end{cases}
$$
\end{lemma}
\begin{proof}
A simple verification of slopes as the one in Proposition \ref{stab:linebundles} will give us $\calo_{X}(2)\in\mathcal{A}^{0,\alpha}$ for all $\alpha^2<2$. Moreover, $\calo_X(2)$ is stable for each stability condition in the $(\alpha,s)$-slice with $\alpha^2<2$. 

Now, notice that
\begin{align*}
    \Ext^i(\calo_X[1],E)&\cong \Ext^{3-i}(E,\calo_X(-2)[1])^*\\
    &=\Ext^{4-i}(\calo_X(2),E^{\vee})^*\\
    &=\Ext^{2-i}(\calo_X(2),E^D)^*.
\end{align*}
Since $E^D$ is again a stable $\mathcal{C}$-instanton object and $\lambda_{0,\alpha,s}(\calo_X(2))>0$ on the region $\mathcal{U}$, then 
$$
\Ext^i(\calo_X[1],E)=0\ \ \text{for}\ \ i\geq 2.
$$
Likewise, $\Ext^i(\calo_X[1],E)=0$ for $i\leq 0$ because $\calo_X[1]$ is stable for every stability condition in the $(\alpha,s)$-slice and $\lambda_{0,\alpha,s}(\calo_X[1])=\lambda_{0,\alpha,s}(E)=0$. Thus
$$
\dim\Ext^{1}(\calo_X[1],E)=\chi(E)=D-\frac{R}{H^3}.
$$
\end{proof}
\begin{proposition}
Let $E$ be an stable $\mathcal{C}$-instanton object, then there exists a $\mathcal{C}$-instanton object $\tilde{E}$ fitting into an exact sequence
$$
0\lra E\lra \tilde{E} \lra \calo_X^{\oplus D-\frac{R}{H^3}}[1]\lra 0,
$$
and satisfying $\Hom^{\bullet}(\calo_X[1],\tilde{E})=0$.
\end{proposition}
\begin{proof}
Since $\Hom^{\bullet}(\calo_X[1],E)\otimes \calo_X[1]=\calo_X^{\oplus D-\frac{R}{H^3}}$ by Lemma \ref{vanishing:Cinst}, then $\tilde{E}$ is just the cone of the evaluation map
$$
\Hom^{\bullet}(\calo_X[1],E)\otimes \calo_X[1]\rightarrow E.
$$
The semistability of $\tilde{E}$ follows from the facts that 
\begin{equation}\label{acyclic:Cinst}
0\lra E\lra \tilde{E} \lra \calo_X^{\oplus D-\frac{R}{H^3}}[1]\lra 0
\end{equation}
is a short exact sequence in $\mathcal{A}^{0,\alpha}$ whenever $E\in \mathcal{A}^{0,\alpha}$, and that
$$
\lambda_{0,\alpha,s}(E)=\lambda_{0,\alpha,s}(\calo_X[1])=0.
$$
The vanishing of $\Hom^{\bullet}(\calo_X[1],\tilde{E})$ follows from applying this functor to the short exact sequence \eqref{acyclic:Cinst}.
\end{proof}
\begin{example}
Assume that $H^3>1$ and let $\ell\subset X$ be a line. Then $\calo_{\ell}$ is a rank 0 instanton sheaf and also a $\mathcal{C}$-instanton. In fact, $\calo_{\ell}$ is a Gieseker stable sheaf and so it is $\lambda_{0,\alpha,s}$-semistable for all $\alpha\gg 0$ (see \cite[Proposition 5.3]{JMM}), and if $\calo_{\ell}$ was ever unstable in the $(\alpha,s)$-slice then it would have to be destabilized by a sub-object $A\hookrightarrow \calo_{\ell}$ with $v_0(A)=(r,c,d,e)$ satisfying 
$$
2d=1,\ \text{and}\ \ 0<c(6e)\leq 1,
$$
which is impossible since $H^3>1$. The acyclic extension of $\calo_{\ell}$ is exactly $\mathcal{I}_{\ell}[1]$, which is a strictly semistable $\mathcal{C}$-instanton (see Example \ref{ex:idealsheaf}). Notice that unlike the case of rank 2 instanton bundles whose acyclic extensions are sheaves, the acyclic extension of a rank 0 instanton sheaf that is also an stable $\mathcal{C}$-instanton is necessarily a complex. 

\end{example}
%%%%%%%%%%%%%%%%%%%%%%%%%%%%%%%%%%%%%%%%%%%%%%%%%%%%%%%%%%%%%%%%%%%

\end{document}